\documentclass{amsart} 

\usepackage{amssymb}
\usepackage{amsthm}
\usepackage{amsfonts}
\usepackage{amsmath}
\usepackage{pmboxdraw}
\usepackage{verbatim} 
\usepackage{graphicx}
\usepackage{color}
\usepackage[colorlinks=true, citecolor=blue, filecolor=black, linkcolor=black, urlcolor=black]{hyperref}
\usepackage{cite}
\usepackage[normalem]{ulem}
\usepackage{subcaption}
\usepackage{bbm}
\usepackage{bm}
\usepackage{mathtools}
\usepackage{todonotes}
\usepackage{kantlipsum}
\usepackage{multirow,makecell}  
\usepackage{mathrsfs}
\allowdisplaybreaks

\newcommand{\RN}[1]{%
	\textup{\uppercase\expandafter{\romannumeral#1}}%
}

\def\C{\mathbb{C}}
\def\D{\mathbb{D}}

\def\P{\mathbf{P}}
\def\R{\mathbb{R}}

\def\cK{\mathcal{K}}

\newcommand{\Pf}{{\textup{Pf}}}
\newcommand{\erfc}{\operatorname{erfc}}
\newcommand{\erf}{\operatorname{erf}}
\newcommand{\bfR}{\mathbf{R}}

\newcommand{\bfK}{\mathbf{K}}
\newcommand{\bfkappa}{{\bm \varkappa}}
\newcommand{\bfG}{{\bm G}}

\newcommand{\Ai}{\operatorname{Ai}}

\newcommand{\re}{\operatorname{Re}}
\newcommand{\im}{\operatorname{Im}}

\theoremstyle{plain}
\newtheorem{thm}{Theorem}[section]

\newtheorem{lem}[thm]{Lemma}

\newtheorem{prop}[thm]{Proposition}

\theoremstyle{remark}

\newtheorem{rem}[thm]{Remark}

\numberwithin{equation}{section}

\begin{document}

\title[Scaling limits of non-Hermitian Wishart ensembles]{Scaling limits of complex and symplectic \\ non-Hermitian Wishart ensembles}
\author{Sung-Soo Byun}
\address{Department of Mathematical Sciences and Research Institute of Mathematics, Seoul National University, Seoul 151-747, Republic of Korea}
\email{sungsoobyun@snu.ac.kr}

\author{Kohei Noda}
\address{Joint Graduate School of Mathematics for Innovation, Kyushu University, West Zone 1, 744 Motooka, Nishi-ku, Fukuoka 819-0395, Japan}
\email{noda.kohei.721@s.kyushu-u.ac.jp}

\begin{abstract}
Non-Hermitian Wishart matrices were introduced in the context of quantum chromodynamics with a baryon chemical potential. These provide chiral extensions of the elliptic Ginibre ensembles as well as non-Hermitian extensions of the classical Wishart/Laguerre ensembles. In this work, we investigate eigenvalues of non-Hermitian Wishart matrices in the symmetry classes of complex and symplectic Ginibre ensembles. We introduce a generalised Christoffel-Darboux formula in the form of a certain second-order differential equation, offering a unified and robust method for analyzing correlation functions across all scaling regimes in the model. By employing this method, we derive universal bulk and edge scaling limits for eigenvalue correlations at both strong and weak non-Hermiticity.
\end{abstract}

\date{\today}

\thanks{Sung-Soo Byun was partially supported by the POSCO TJ Park Foundation (POSCO Science Fellowship) and by the New Faculty Startup Fund at Seoul National University. Kohei Noda was partially supported by WISE program (JSPS) at Kyushu University, JSPS KAKENHI Grant Number (B) 18H01124 and 23H01077, and the Deutsche Forschungsgemeinschaft (DFG) grant SFB 1283/2 2021–317210226.
}

\maketitle

\section{Introduction and main results}

\subsection{Non-Hermitian Wishart matrices}

For given non-negative integers $N$ and $\nu$, we denote by $\textup{P}$ and $\textup{Q}$ the $N \times (N+\nu)$ random matrices with independent complex or quaternionic Gaussian entries of mean $0$ and variance $1/(2N)$. 
These are known as the complex or symplectic rectangular Ginibre ensembles, respectively, see \cite{BF22,BF23} for recent reviews.   
For a given non-Hermiticity  parameter $\tau \in [0,1]$, the non-Hermitian Wishart matrix is defined by 
\begin{equation}\label{def of nWishart}
X:=X_1 X_2^*, \qquad \begin{cases}
X_1=\sqrt{1+\tau} \,\, \textup{P}+\sqrt{1-\tau} \,\, \textup{Q},
\smallskip 
\\
X_2=\sqrt{1+\tau} \, \, \textup{P}-\sqrt{1-\tau} \,\, \textup{Q}. 
\end{cases}
\end{equation} 
For the external case $\tau=0$, the maximally non-Hermitian regime, the model gives a product of two rectangular Ginibre matrices, while in the opposite extremal case $\tau=1$, the Hermitian regime, the model coincides with the Laguerre unitary or symplectic ensemble (LUE or LSE) \cite{Fo10}, respectively. 
We mention that the model \eqref{def of nWishart} is also referred to as the chiral Ginibre ensemble \cite{Ste96,Os04}, where these models find applications in analyzing the Dirac operator spectrum of quantum chromodynamics with a chemical potential. It is also called the sample cross-covariance matrices \cite{BBD23} and finds applications for the analysis of time series \cite{KS10}.

We study eigenvalues of non-Hermitian complex/symplectic Wishart ensembles. 
As a common feature of integrable random matrices, the model \eqref{def of nWishart} also enjoys a statistical physics realisation in a sense that the ensemble can be interpreted as the Coulomb gas at $\beta=2$ with Dirichlet (complex) or Neumann (symplectic) boundary conditions along the real axis \cite{Os04,Ak05,Fo10}. 
To be more concrete, let us first setup some notations.
For given parameters $A>B \ge 0$, we consider the weight functions 
\begin{equation} \label{weight functions}
\omega^{ \rm c }_N(z) :=  |z|^{\nu} K_{\nu}(AN|z|) e^{ BN \re z }, \qquad  \omega^{ \rm s }_N(z) := |z|^{2\nu} K_{2\nu}(2AN|z|) e^{ 2BN \re z }, 
\end{equation} 
where $K_{\nu}$ is the modified Bessel function of the second kind
\begin{equation}
K_{\nu}(z)=\frac{\pi}{2}\frac{I_{-\nu}(z)-I_{\nu}(z)}{\sin(\nu\pi)},\qquad
I_{\nu}(z)=\sum_{k=0}^{\infty}\frac{(z/2)^{2k+\nu}}{k! \, \Gamma(k+\nu+1)}.
\end{equation}
The parameters $A$ and $B$ are related with the non-Hermiticity parameter $\tau \in [0,1)$ as 
\begin{equation} \label{def of A and B}
A=\frac{2}{1-\tau^2},\qquad
B=\frac{2\tau}{1-\tau^2}, 
\end{equation}
Then the eigenvalues $\boldsymbol{z}=(z_1,\cdots,z_N) \in \C^N$ of the non-Hermitian Wishart matrix \eqref{def of nWishart} follow the joint probability distribution
\begin{align}
d \P_N^{\rm c}(\boldsymbol{z}) & :=\frac{1}{Z_N^{\rm c} } \prod_{j>k=1}^N |z_j-z_k|^{2} \prod_{j=1}^N \omega^{ \rm c }_N(z_j) \, dA(z_j), \label{Gibbs complex}
\\
d \P_N^{ \rm s }(\boldsymbol{z}) & :=\frac{1}{Z_N^{ \rm s } } \prod_{j>k=1}^N |z_j-z_k|^{2} |z_j-\overline{z}_k|^2 \prod_{j=1}^N  |z_j-\overline{z}_j|^2 \omega^{ \rm s }_N(z_j)\, dA(z_j), \label{Gibbs symplectic}
\end{align}
respectively, where $dA(z)=d^2z/\pi$ is the area measure, and $Z_N^{\rm c}$ and $Z_N^{ \rm s } $ are the partition functions that make \eqref{Gibbs complex} and \eqref{Gibbs symplectic} probability measures.
Contrary to the matrix model realisation, for the ensembles \eqref{Gibbs complex} and \eqref{Gibbs symplectic}, one can consider general real-valued $\nu$, as long as $\nu > -1$ for the complex case and $\nu > -1/2$ for the symplectic case. 
Note that a characteristic difference between \eqref{Gibbs complex} and \eqref{Gibbs symplectic} is the repulsion along the real axis that comes from the factor proportional to $|z_j-\overline{z}_j|^2$, see Figure~\ref{Fig_droplet}.  

Due to the above Boltzmann-Gibbs measure realisations together with the standard equilibrium convergence (see e.g. \cite{BC12,Am21,CGZ14}), the macroscopic behaviours of $\boldsymbol{z}$ can be described using the logarithmic potential theory \cite{ST97}. 
The associated planar equilibrium measure problem was solved in \cite{ABK21}. 
Consequently, it follows that for any fixed $\nu$, the empirical measure $\frac{1}{N} \sum_{j=1}^N \delta_{ z_j }$ weakly converges to  
\begin{equation} \label{MeasMu}
d\mu(z) = \frac{1}{2(1-\tau^2)}  \frac{1}{ |z|  }\cdot \mathbbm{1}_{ S }(z) \, dA(z).
\end{equation}
Here, the limiting spectrum $S$ is closed by the ellipse 
\begin{equation} \label{S droplet}
S := \Big\{ (x,y) \in \R^2:  \Big( \frac{x-2\tau   }
	{ 1+\tau^2  }   \Big)^2+\Big( \frac{y}{ 1-\tau^2 }  \Big) ^2 \le 1 \Big\},
\end{equation}
see Figure~\ref{Fig_droplet}.
This limiting distribution \eqref{MeasMu} provides a non-Hermitian generalisation of the classical Marchenko-Pastur law. 
We mention that the results in \cite{ABK21} indeed cover the case $\nu=O(N)$ as well. 
We also refer to \cite{Ka23} for the universality of the limiting law \eqref{MeasMu} for the square ($\nu=0$) and complex case.

\begin{figure}[t]
	\begin{subfigure}{0.48\textwidth}
		\begin{center}	
			\includegraphics[width=\textwidth]{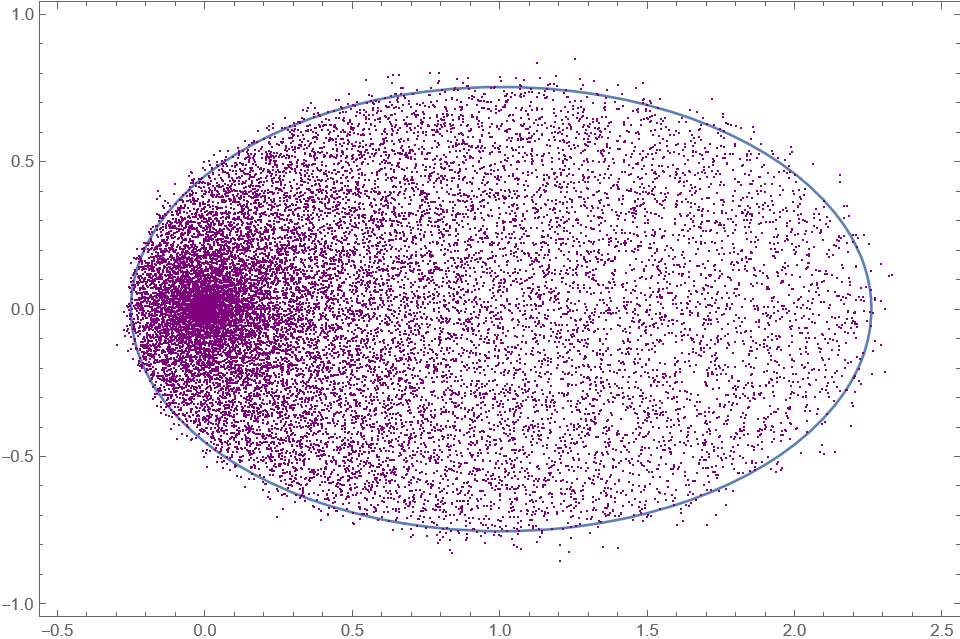}
		\end{center}
		\subcaption{Complex}
	\end{subfigure}	
 \quad
	\begin{subfigure}{0.48\textwidth}
		\begin{center}	
			\includegraphics[width=\textwidth]{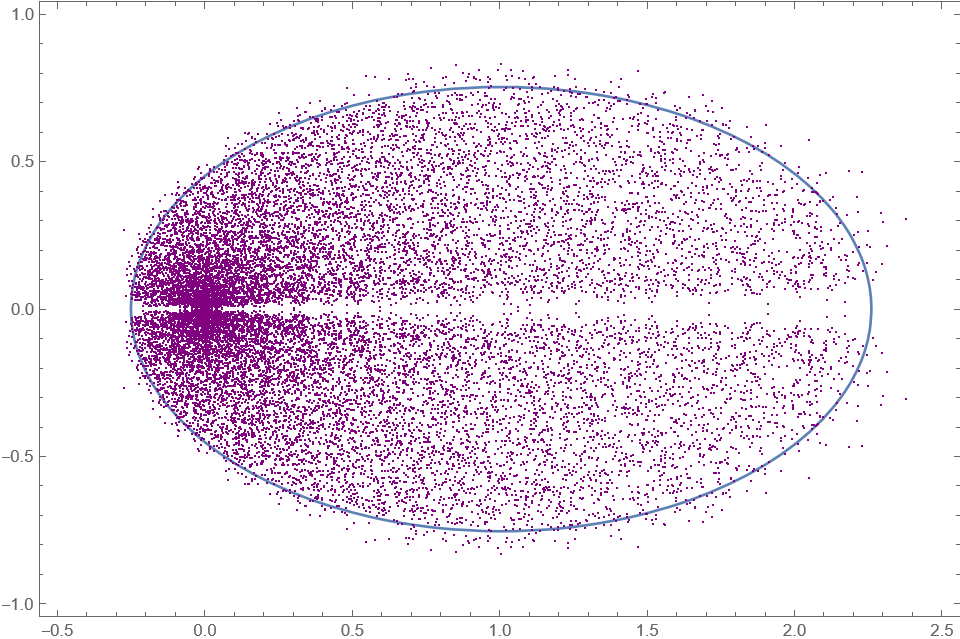}
		\end{center}
		\subcaption{Symplectic}
	\end{subfigure}	
	\caption{100 realisations of eigenvalues of the complex and symplectic non-Hermitian Wishart matrices, where $\tau=0.5$, $\nu=1$ and $N=200$. Here, the full line represents the ellipse given in \eqref{S droplet}. 
    In both (A) and (B), the accumulation of eigenvalues near the origin is evident, cf. \eqref{MeasMu}. Additionally, in (B), one can observe repulsion along the real axis. } \label{Fig_droplet}
\end{figure}

\medskip 

Beyond the averaged eigenvalue density, more detailed statistical behaviours of the eigenvalues are encoded in the $k$-point correlation functions 
\begin{align}
\bfR_{N,k}^{ \rm c }(z_1,\dots,z_k)&:= \frac{N!}{(N-k)!} \int_{ \C^{ N-k } } \P_N^{ \rm c }( \boldsymbol{z} ) \prod_{j=k+1}^N \,dA(z_j),  \label{def of RNk complex}
\\
\bfR_{N,k}^{ \rm s }(z_1,\dots,z_k)&:= \frac{N!}{(N-k)!} \int_{ \C^{ N-k } } \P_N^{ \rm s }( \boldsymbol{z} ) \prod_{j=k+1}^N \,dA(z_j). \label{def of RNk symplectic}
\end{align}
It is well known that the ensemble \eqref{Gibbs complex} forms a determinantal point process whose correlation kernel can be constructed using planar orthogonal polynomials. 
Similarly, the ensemble \eqref{Gibbs symplectic} forms a Pfaffian point process, and in this case the associated correlation kernel can be constructed using planar skew-orthogonal polynomials.

It is a surprising fact, especially given the seemingly complicated forms of the weight functions \eqref{weight functions}, that one can express the associated (skew)-orthogonal polynomials in terms of the generalised Laguerre polynomial
\begin{equation}
L_j^{(\nu)}(z):=\sum_{k=0}^{j}\frac{\Gamma(j+\nu+1)}{(j-k)!\Gamma(\nu+k+1)}\frac{(-z)^k}{k!}. 
\end{equation}
Note that this shows a common feature with the extremal case when $\tau=1$, the LUE and LSE \cite{AFNV00}.  
As a consequence, by combining the results in \cite{Os04,Ak05,AEP22}, the correlation kernels can be expressed in terms of the Laguerre polynomials, see Subsection~\ref{Subsec_Laguerre} for more details. 

\medskip 

In this wok, we investigate various scaling limits of correlation functions \eqref{def of RNk complex} and \eqref{def of RNk symplectic}. 
For this purpose, for a given zooming point $p \in S \setminus \{0\}$, let
\begin{equation} \label{def of delta Lap Q}
\delta \equiv \delta(p) = \frac{1}{2(1-\tau^2)|p|} 
\end{equation}
be the mean eigenvalue density, cf. \eqref{MeasMu}. Then we define the rescaled $k$-point correlation functions 
\begin{align}
R_{N,k}^{ \rm c }(z_1,\cdots,z_k)& =  \displaystyle \frac{1}{(N\delta)^{k}} \bfR_{N,k}^{ \rm c }\Big(p+ \frac{\mathbf{n}(p)\, z_1}{ \sqrt{N \delta} },\cdots, p+ \frac{\mathbf{n}(p)\, z_k}{ \sqrt{N \delta} } \Big),    \label{RNk rescaled complex}
\\
R_{N,k}^{ \rm s }(z_1,\cdots,z_k)& =  \displaystyle \frac{1}{(N\delta)^{k}} \bfR_{N,k}^{ \rm s }\Big(p+ \frac{\mathbf{n}(p)\, z_1}{ \sqrt{N \delta} },\cdots, p+ \frac{\mathbf{n}(p)\, z_k}{ \sqrt{N \delta} } \Big),    \label{RNk rescaled symplectic}
\end{align}
where the $\mathbf{n}(p)$ is the outward unit normal vector from $\partial S$ at $p$, and otherwise, $\mathbf{n}(p)=1$ for $p\in\mathrm{int}(S)$.

In the asymptotic analysis of the non-Hermitian Wishart models, one needs to distinguish the following two different regimes.
\begin{itemize}
    \item[(i)] \textbf{Strong non-Hermiticity.} This is the case that $\tau \in [0,1)$ is fixed. In this case, the limiting spectrum is given by \eqref{S droplet}, which is a genuinely two-dimensional subset of $\C$.
    \smallskip 
    \item[(ii)] \textbf{Weak non-Hermiticity.} This is the case that $\tau \uparrow 1$ as $N\to \infty$ with a proper speed depending on the position of the zooming point. This critical regime gives rise to an interpolation between the one- and two-dimensional point processes. 
\end{itemize}
Furthermore, from the viewpoint of the universality classes, it is also necessary to distinguish cases based on the zooming points where we look at the local statistics.
\begin{itemize}
    \item[(a)] \textbf{Singular origin.} This is the case $p=0$ where the density at the origin diverges \eqref{MeasMu}. Due to this fact, it gives rise to a non-standard universality classes in two-dimensional point process.
    \smallskip 
    \item[(b)] \textbf{Regular bulk.} This is the case where $p$ is the interior of the droplet. For the complex ensemble, one would expect the universality classes  \cite{Gin65,AHM11,FKS97,ACV18} to arise, and for the symplectic ensemble, one would expect the universality classes in \cite{Kan02, BES23, AKMP19}. 
    \smallskip 
    \item[(c)] \textbf{Regular edge.} This is the case where $p$ is on the boundary of the droplet. 
    In this regime, one would expect the universality classes in \cite{FH99,HW21, CES21, Be10, ACC23a} for the complex ensemble, and those in \cite{ABK22, AP14, BES23} for the symplectic ensemble. 
\end{itemize}
In summary, there are $6$ different local universality classes for each complex and symplectic non-Hermitian Wishart ensembles. Among these 12 different regimes, 6 of them have been investigated in the literature \cite{Os04, Ak05, AB10, AEP22, AB23}, see Table~\ref{Table_summary}. Let us briefly summarize the previous development in this direction.

\begin{itemize}
    \item The \textbf{origin} scaling limit of the \textbf{complex} ensemble at \textbf{strong} non-Hermiticity is equivalent to the well-known Hardy-Hille formula 
    \begin{equation} \label{HH formula}
\sum_{j=0}^{\infty}\frac{j!\,\tau^{2j}}{\Gamma(j+\nu+1)} L_{j}^{(\nu)}\Bigl(\frac{z}{\tau}\Bigr) L_{j}^{(\nu)}\Bigl(\frac{w}{\tau}\Bigr) =\frac{ e^{-\frac{\tau}{1-\tau^2}(z+w) } }{(1-\tau^2)(zw)^{\nu/2}}I_{\nu}\Bigl(\frac{2\sqrt{zw}}{1-\tau^2}\Bigr).
\end{equation}
    \item In \cite{Os04}, Osborn derived the \textbf{origin} scaling limit of the \textbf{complex} ensemble at \textbf{weak} non-Hermiticity. 
    The resulting process provides a non-Hermitian extension of the hard edge Bessel point process of the LUE. 
    For this, the Bessel $J_\nu$ asymptotic formula of the Laguerre polynomial at the origin and the Riemann sum approximation were used. 
    \smallskip 
    \item In \cite{Ak05}, Akemann obtained the \textbf{origin} scaling limit of the \textbf{symplectic} ensemble at \textbf{weak} non-Hermiticity. 
    This provides a non-Hermitian extension of the hard edge Bessel point process of the LSE. 
    For this, a certain differential equation for the large-$N$ limit of the kernel was found and utilized.
    \smallskip 
    \item In \cite{AB10}, Akemann and Bender obtained the \textbf{edge} scaling limit of the \textbf{complex} ensemble at \textbf{weak} non-Hermiticity in the Dirac picture \eqref{def of Dirac}. 
    The resulting process coincides with a non-Hermitian extension of the soft edge Airy point process of the GUE previously found in \cite{Be10} for the elliptic GinUE. 
    In this case, the scaling limit was obtained using a double contour integral representation of the kernel and the steepest descent analysis.
    \smallskip 
    \item In \cite{AEP22}, Akemann, Ebke and Parra obtained the \textbf{origin} scaling limit of the \textbf{symplectic} ensemble at \textbf{strong} non-Hermiticity. 
    In this case, a double contour integral formula of the kernel and the differential equation found in \cite{Ak05} to investigate the maximally non-Hermitian regime ($\tau=0$) were utilized. 
    \smallskip 
    \item In \cite{AB23}, Ameur and Byun obtained the \textbf{bulk} scaling limit of the \textbf{complex} ensemble at \textbf{weak} non-Hermiticity. 
    The resulting process coincides with a non-Hermitian extension of the sine point process of the GUE introduced in \cite{FKS97,FKS97a,FKS98} for the elliptic GinUE. 
    In \cite{AB23}, the authors made use of the theory of Ward's equations \cite{AKM19,AKMW20} together a contour integral representation of the kernel. 
\end{itemize}

\begin{table}[h!]
    \centering
   {\def\arraystretch{2}
\begin{tabular}{ |p{1.5cm}|p{1.5cm}|p{1.5cm}|p{1.5cm}|  }
\hline
\multicolumn{4}{|c|}{\textbf{Complex Non-Hermitian ensemble} } 
\\
\hline
& \centering \textup{Origin} & \centering \textup{Bulk} &  \hspace{0.35cm}\textup{Edge} 
\\ 
\hline 
\centering Strong & \centering \eqref{HH formula} & \centering  Unknown & Unknown
 \\ 
 \hline
\centering Weak & \centering \cite{Os04} & \centering \cite{AB23} & \hspace{0.35cm} \cite{AB10} 
 \\ 
   \hline 
\end{tabular}
\qquad 
\begin{tabular}{ |p{1.5cm}|p{1.5cm}|p{1.5cm}|p{1.5cm}|  }
\hline
\multicolumn{4}{|c|}{\textbf{Symplectic Non-Hermitian ensemble} } 
\\
\hline
& \centering \textup{Origin} & \centering \textup{Bulk} &  \hspace{0.35cm}\textup{Edge} 
\\ 
\hline 
\centering Strong & \centering \cite{AEP22} & \centering Unknown  & Unknown
 \\ 
 \hline
\centering Weak & \centering \cite{Ak05} & \centering Unknown & Unknown
 \\ 
   \hline 
\end{tabular}
}
\caption{Summary of the previous results on scaling limits of complex/symplectic non-Hermitian Wishart ensembles. All of these unknown regimes are addressed in Theorems~\ref{Thm_StrongNonHermiticity}, ~\ref{Thm_BulkWeakNonHermiticity} and ~\ref{Thm_EdgeWeakNonHermiticity} below. 
}
\label{Table_summary}
\end{table}

Other than the aforementioned cases, the scaling limits of non-Hermitian Wishart ensembles remain undiscovered, despite their intrinsic interest, particularly from the viewpoint of the universality principle \cite{Ku11}. 
In this work, we introduce a unified framework for analyzing the correlation functions (Theorem~\ref{Thm_ODE}) for finite-$N$. 
As a consequence, we obtain the scaling limits of the complex and symplectic non-Hermitian Wishart ensembles for all unknown regimes, i.e. both strong and weak non-Hermiticity, as well as both bulk and edge scaling regimes (Theorems~\ref{Thm_StrongNonHermiticity}, \ref{Thm_BulkWeakNonHermiticity}, and \ref{Thm_EdgeWeakNonHermiticity}).
Indeed, our method can also be employed to re-derive the results in \cite{Os04, Ak05, AB10, AEP22, AB23} in a unified manner. An additional advantage of our method is that it does not require $\nu$ to be an integer, unlike the method using the contour integral representation.
Furthermore, compared to the previous methods, our method is more easy to track the error terms, yielding finite-size corrections as well.

\subsection{Main results}

In this section, we introduce our main results.

Due to the lack of the classical Christoffel-Darboux formula for non-Hermitian random matrix ensembles, one should find a proper way of analyzing the correlation kernel. 
Our first result provides a differential equation satisfied by the correlation kernels of the Wishart ensembles. 
As will become clear, this turns out to be very helpful for proceeding with further asymptotic analysis.

\begin{thm}[\textbf{Differential equations for correlation kernels of planar Wishart ensembles}] \label{Thm_ODE}
For any $\tau \in [0,1]$, $\nu \in \R$ and $N \in \mathbb{N}$, we have the following. 
\begin{itemize}
    \item[(i)] \textbf{\textup{(Complex ensemble)}} 
Let 
\begin{equation}  \label{def of mathcalK}
\mathcal{K}_N(z,w) \equiv \mathcal{K}_N^{(\nu)}(z,w):=  \sum_{j=0}^{N-1}\frac{j!\, \tau^{2j} }{\Gamma(j+\nu+1)}
L_{j}^{(\nu)}\Bigl(\frac{z}{\tau}\Bigr) L_{j}^{(\nu)}\Bigl(\frac{w}{\tau}\Bigr) . 
\end{equation}
Then we have 
\begin{align}
\begin{split} 
&\quad \bigg[ (1-\tau^2)\,z\,\partial_{z}^2  +\Big((1-\tau^2)(\nu+1)+2\tau z \Big)\partial_{z} + \frac{ \tau^2 z -w}{1-\tau^2}+(\nu+1)\tau \bigg]\cK_N^{(\nu)}(z,w)
\\
&= \frac{N!}{\Gamma(N+\nu)} \frac{\tau^{2N-1}}{1-\tau^2} 
 \bigg[ 
L_{N-1}^{(\nu)}\Bigl(\frac{z}{\tau}\Bigr)L_{N}^{(\nu)}\Bigl(\frac{w}{\tau}\Bigr)
-
\tau^2L_{N}^{(\nu)}\Bigl(\frac{z}{\tau}\Bigr)L_{N-1}^{(\nu)}\Bigl(\frac{w}{\tau}\Bigr)
 \bigg]. 
\end{split}
\end{align}

\smallskip 
\item[(ii)] \textbf{\textup{(Symplectic ensemble)}} 
Let 
\begin{equation} \label{def of kappaN main}
\kappa_N(z,w) \equiv \kappa_N^{(\nu)}(z,w) := G_N(z,w)-G_N(w,z), 
\end{equation}
where 
\begin{equation}
G_N(z,w):= -  \frac{\sqrt{\pi}}{2^{2\nu+1}} \sum_{k=0}^{N-1}\sum_{j=0}^{k}\frac{(2k)!!\,\tau^{2k+1}}{2^k\Gamma(k+\nu+3/2)}
\frac{(2j-1)!!\,\tau^{2j}}{2^j\Gamma(j+\nu+1)}L_{2k+1}^{(2\nu)}\Bigl(\frac{z}{\tau}\Bigr)L_{2j}^{(2\nu)}\Bigl(\frac{w}{\tau}\Bigr).
\end{equation}
Then we have 
\begin{align}
\begin{split} \label{ODE for kappaN}
&\quad \bigg[(1-\tau^2)z\, \partial_{z}^2
+\Big( (2\nu+1)(1-\tau^2)+2\tau z \Big)\partial_{z}
+\tau(2\nu+1)-z \bigg]  \kappa_N^{(\nu)}(z,w)
\\ 
& =  \cK_{2N}^{(2\nu)}(z,w)  - \frac{  \sqrt{\pi}\, N!\, \tau^{2N}}{2^{2\nu} \Gamma(N+\nu+1/2)} \,L_{2N}^{(2\nu)}\Bigl(\frac{z}{\tau} \Bigr)
\sum_{j=0}^{N-1}\frac{(2j-1)!!\tau^{2j}}{2^j\Gamma(j+\nu+1)}L_{2j}^{(2\nu)}\Bigl(\frac{w}{\tau}\Bigr).
\end{split}
\end{align}
\end{itemize}
\end{thm}

Up to a proper scaling, $\mathcal{K}_N$ in \eqref{def of mathcalK} is the correlation kernel of the complex Wishart ensemble, forming a determinantal point process. Similarly, $\kappa_N$ in \eqref{def of kappaN main} is a pre-kernel of a $2\times 2$ matrix-valued kernel of the symplectic Wishart ensemble, forming a Pfaffian point process, see Proposition~\ref{Prop_Det Pfa struc}.

From the viewpoint of the special function theory, Theorem~\ref{Thm_ODE} provides certain functional identities involving the Laguerre polynomials, which are new to our best knowledge. 
We stress that it is indeed far from being obvious that there exists such functional identities.

The idea of constructing a differential equation for the correlation kernel was employed by Lee and Riser in their study of the elliptic GinUE \cite{LR16}, see also \cite{BLY21}. 
Furthermore, this method proves particularly helpful in the asymptotic study of the planar symplectic ensembles. 
We refer to \cite{Kan02,ABK22} for the implementation in the study of the GinSE, \cite{BE23} for its extension to the elliptic GinSE, \cite{BC23} for the induced GinSE, \cite{BF23} for the truncated ensemble, and \cite{BF23a} for the induced spherical GinSE.
In all of these works, a specific first-order differential equation for each model was constructed. 

It is also worth highlighting that in the study of planar symplectic ensembles \cite{ABK22, BE23, BC23, BF23a}, the inhomogeneous term in the differential equation involves the correlation kernel of their complex counterparts. 
This common feature can also be observed in \eqref{ODE for kappaN}, where the term $\mathcal{K}_{2N}^{(2\nu)}$ appears. 
Such a relation between correlation functions of unitary and symplectic ensembles was also found in \cite{AFNV00,Wi99} for Hermitian random matrix ensembles, see also \cite{Li23} for a recent work.   
This relation was crucially used for the universality of Hermitian random matrix ensembles with symplectic symmetry \cite{DG07a,DG07b}.

\bigskip 

We now turn to the scaling limits of the non-Hermitian Wishart ensembles. Let us first discuss the strong non-Hermitian regime. Note that, by \eqref{S droplet}, the intersection of the droplet $S$ with the real axis is given by 
\begin{equation}
S \cap \R = [e_-,e_+ ],  \qquad e_{\pm}:= 2\tau \pm (1+\tau^2). 
\end{equation}
In the asymptotic analysis of planar symplectic ensembles, the behaviour varies depending on proximity to the real axis. 
Namely, the self-interaction term $|z_j-\overline{z}_j|^2$ in \eqref{Gibbs symplectic} does not affect the local behaviour away from the real axis in the large $N$ limit. 
We shall focus on the real axis case for the planar symplectic ensemble. 
On the other hand, for the complex ensemble, there is no speciality of the real axis and we consider the general zooming points in the complex plane.
In summary, we shall use the following terminology. 
\begin{itemize}
    \item Bulk case: $p \in \textup{int}(S) \setminus \{0\}$ for the complex ensemble, $p \in (e_-,e_+) \setminus \{0\}$ for the symplectic ensemble. 
    \smallskip 
    \item Edge case: $p \in \partial S$ for the complex ensemble, $p \in \{ e_-,e_+\}$ for the symplectic ensemble.
\end{itemize}
We have the following results, which demonstrate the universal scaling limits.

\begin{thm}[\textbf{Bulk/edge scaling limits of planar Wishart ensembles at strong non-Hermiticity}]
\label{Thm_StrongNonHermiticity}
Let $\tau \in [0,1)$ be fixed. 
Then we have the following. 
\begin{itemize}
    \item[(i)] \textbf{\textup{(Complex ensemble)}} For any $\nu>-1$, we have 
\begin{equation}
\lim_{N \to \infty} R_{N,k}^{ \rm c }(z_1,\dots,z_k) = \det \Big[ K(z_j,z_l) \Big]_{j,l=1}^k, 
\end{equation}
uniformly on compact subsets of $\C$, where 
\begin{align}
 \label{GinUE bulk kernel}
&K(z,w) = K_{ \rm b }(z,w):=  e^{ z\overline{w}-\frac{|z|^2}{2}-\frac{|w|^2}{2} }, \qquad \textup{for the bulk case,}
\\
&  \label{GinUE edge kernel}
K(z,w) = K_{ \rm e }(z,w):= e^{ z\overline{w}-\frac{|z|^2}{2}-\frac{|w|^2}{2} } \frac12 \erfc\Big( \frac{z+\overline{w}}{ \sqrt{2} } \Big), \qquad \textup{for the edge case.}
\end{align}
 \item[(ii)] \textbf{\textup{(Symplectic ensemble)}} For any $\nu>-1/2$, we have 
\begin{equation}
\lim_{N \to \infty} R_{N,k}^{ \rm s }(z_1,\dots,z_k) = \prod_{j=1}^k (\overline{z}_j-z_j) \, \Pf \bigg[ e^{-|z_j|^2-|z_l|^2} 
\begin{pmatrix} 
\kappa(z_j,z_l) & \kappa(z_j,\overline{z}_l)
\smallskip 
\\
\kappa(\overline{z}_j,z_l) & \kappa(\overline{z}_j,\overline{z}_l) 
\end{pmatrix} 
\bigg]_{j,l=1}^k, 
\end{equation}
uniformly on compact subsets of $\C$, where 
\begin{align}
 \label{GinSE bulk kernel}
&	\kappa(z,w)= \kappa_{ \rm b }(z,w) := \sqrt{\pi} e^{z^2+w^2} \erf(z-w), \qquad \textup{for the bulk case},
\\
& \label{GinSE edge kernel}
\kappa(z,w) = \kappa_{ \rm e }(z,w) :=\sqrt{\pi} e^{z^2+w^2} \int_{-\infty}^0 W(f_{w},f_{z})(u) \, du,  \qquad \textup{for the edge case}.
\end{align} 
Here $W(f,g):=fg'-gf'$ is the Wronskian and $f_z(u):=\frac12 \erfc(\sqrt{2}(z-u)). $ 
\end{itemize}
\end{thm}

Note that $K_{ \rm b }$ in \eqref{GinUE bulk kernel} is the limiting bulk kernel of the GinUE \cite{Gin65}, whereas $K_{ \rm e }$ in \eqref{GinUE edge kernel} is the limiting edge kernel of the GinUE \cite{FH99}. 
We stress that the bulk universality of random normal matrices (i.e. \eqref{Gibbs complex} with a general weight function) was shown in \cite{AHM11}. On the other hand, the edge universality was shown in \cite{HW21}. 
Nonetheless, Theorem~\ref{Thm_StrongNonHermiticity} (i) is not a direct consequence of \cite{AHM11,HW21}, as the weight function \eqref{weight functions} depends highly on $N.$

Regarding the symplectic ensemble, note that $\kappa_{ \rm b }$ in \eqref{GinSE bulk kernel} coincides with the limiting bulk kernel of the GinSE \cite{Kan02,AKMP19}. Similarly, $\kappa_{ \rm e }$ in \eqref{GinSE edge kernel} corresponds to the limiting edge kernel of the GinSE recently found in \cite{ABK22}. In contrast to the random normal matrix ensemble, the universality of the planar symplectic ensemble (i.e. \eqref{Gibbs symplectic} with a general weight function) is much less developed. 
To be more precise, the universal scaling limits \eqref{GinSE bulk kernel} and \eqref{GinSE edge kernel} have been shown only for the GinSE \cite{Kan02,ABK22}, elliptic GinSE \cite{AEP22,BE23}, spherical GinSE \cite{BF23a}, induced GinSE \cite{BC23,BF23}, and truncated ensembles \cite{KL21,BF23}. 
The fundamental reason for this relatively slow progress compared to the random normal matrix model is that there is no general theory on the skew-orthogonal polynomial in the plane. 
Furthermore, once one can construct the skew-orthogonal polynomials, the method for performing proper asymptotic analysis has only recently been developed. 

As a side remark, let us also mention that the limiting correlation kernels of the GinUE and GinSE can be described in a unified manner. 
For this purpose, we write 
\begin{equation}
E:=\begin{cases}
    	    (-\infty,\infty) &\textup{for the bulk case,}
    	    \smallskip 
    	    \\
    	    (-\infty,0) &\textup{for the edge case}.
    	\end{cases}
\end{equation}
Then the GinUE correlation kernels \eqref{GinUE bulk kernel} and \eqref{GinSE edge kernel} can be written as 
\begin{equation}
K(z,w) = e^{ z\overline{w}-\frac{|z|^2}{2}-\frac{|w|^2}{2} }\, \frac{ 1  }{ \sqrt{2\pi} } \int_E e^{ -\frac12( z-\overline{w}-t )^2 }\,dt .
\end{equation} 
Similarly, the GinSE correlation kernels \eqref{GinSE bulk kernel} and \eqref{GinSE edge kernel} can be expressed in terms of the Wronskian form
\begin{equation} \label{kappa Wronskian}
	\kappa(z,w):=\sqrt{\pi} e^{z^2+w^2} \int_{E} W(f_{w},f_{z})(u) \, du. 
\end{equation}  
In these unified formulas, by taking an interval of the form $E=(-\infty,a)$, one can also interpolate bulk and edge scaling limits, see e.g. \cite[Remark 2.3 (iii)]{ABK22}. 

\bigskip 


Next, we discuss the scaling limits at weak non-Hermiticity. 
Our first result gives the bulk scaling limits. For this, let us write 
\begin{equation} \label{MP law}
\sigma_{{\rm MP}}(\xi):=\frac{1}{2\pi }\sqrt{\frac{4-\xi}{\xi}} \,\mathbf{1}_{[0,4]}(\xi)
\end{equation}
for the Marchenko-Pastur law. 
One may expect that the global eigenvalue density of the non-Hermitian Wishart ensemble when $\tau \uparrow 1$ is close to the Marchenko-Pastur law.
We refer to \cite[Section 2]{AB23} for a geometric description on the asymptotic shape of the droplet in the weakly non-Hermitian (or bandlimited Coulomb gas) ensembles. 
For the bulk scaling regime when $p \in (0,4)$, we have the following. 

\begin{thm}[\textbf{Bulk scaling limits of planar Wishart ensembles at weak non-Hermiticity}]
\label{Thm_BulkWeakNonHermiticity}
Let  
\begin{equation}
\label{BulkWeakTau}
 \tau=1- \frac{c^2}{N} , \qquad c \in (0,\infty)  
\end{equation}
and define
\begin{equation}
\label{Bulk Weak ap}
E_a= ( -2a,2a ), \qquad a = a(p)=\pi c\sqrt{p}\,\sigma_{{\rm MP}}(p).
\end{equation}
Suppose that $p \in (0,4)$. Then we have the following. \smallskip 
\begin{itemize}
     \item[(i)] \textbf{\textup{(Complex ensemble)}} For any $\nu>-1$, we have 
\begin{equation} 
\lim_{N \to \infty} R_{N,k}^{ \rm c }(z_1,\dots,z_k) = \det \Big[ K^{c}_{\rm{b}}(z_j,z_l) \Big]_{j,l=1}^k, 
\end{equation}
uniformly on compact subsets of $\C$, where 
\begin{equation} \label{GinUE wH bulk kernel}
K^{c}_{\rm{b}}(z,w) = e^{ z\overline{w}-\frac{|z|^2}{2}-\frac{|w|^2}{2} }\, \frac{ 1  }{ \sqrt{2\pi} } \int_{ E_a } e^{ \frac12( z-\overline{w}-i t )^2 }\,dt .
\end{equation} 
 \item[(ii)] \textbf{\textup{(Symplectic ensemble)}} For any $\nu>-1/2$, we have 
\begin{equation}
\lim_{N \to \infty} R_{N,k}^{ \rm s }(z_1,\dots,z_k) = \prod_{j=1}^k (\overline{z}_j-z_j) \, \Pf \bigg[ e^{-|z_j|^2-|z_l|^2} 
\begin{pmatrix}  
\kappa^c_{ \rm b }(z_j,z_l) & \kappa^c_{ \rm b }(z_j,\overline{z}_l)
\smallskip 
\\
\kappa^c_{ \rm b }(\overline{z}_j,z_l) & \kappa^c_{ \rm b }(\overline{z}_j,\overline{z}_l) 
\end{pmatrix} 
\bigg]_{j,l=1}^k, 
\end{equation}
uniformly on compact subsets of $\C$, where 
\begin{equation}  \label{GinSE wH bulk kernel}
\kappa^c_{ \rm b }(z,w):=	 \frac{1}{\sqrt{\pi}}  e^{ z^2+w^2 }  \int_{E_a} e^{-u^2} \sin(2u(z-w) ) \frac{du}{u}. 
\end{equation} 
\end{itemize}
\end{thm}

The scaling limit \eqref{GinUE wH bulk kernel} was first introduced in the serial work \cite{FKS97, FKS97a, FKS98} studying the elliptic GinUE at weak non-Hermiticity. 
This has been extended in \cite{ACV18} for the fixed trace elliptic GinUE with perturbation.
Note that under the assumption that $\nu$ is an integer, Theorem~\ref{Thm_BulkWeakNonHermiticity} (i) was obtained in \cite[Theorem 1.7]{AB23} using the theory of Ward's identity \cite{AKM19,AKMW20}. 
Let us mention that, unlike the universality of random normal matrices at strong non-Hermiticity \cite{AHM11,HW21}, the universality at weak non-Hermiticity has been less developed. 
Nonetheless, this universality class appears in the almost-circular ensemble \cite{BS23} and near the singular boundary point \cite{AKMW20}. 
See also a recent work \cite{Os23} for a non-integrable model lying in this universality class.

On the other hand, the scaling limit \eqref{GinSE wH bulk kernel} was obtained in the work of Kanzieper \cite{Kan02} studying the weakly non-Hermitian elliptic GinSE at the origin. For the elliptic GinSE, this was extended to the whole real axis in \cite{BES23}, where a certain Wronskian structure was also introduced. However, except for the elliptic GinSE, the existence of the universality class with the kernel \eqref{GinSE wH bulk kernel} has not been known to our knowledge. Therefore, by Theorem~\ref{Thm_BulkWeakNonHermiticity} (ii), we contribute to finding a new example.


\medskip 

Our final results address the edge scaling limits at weak non-Hermiticity. 
For this, recall that the Airy function is defined
\begin{equation}
\Ai(x):= \frac{1}{\pi} \int_0^\infty \cos\Big( \frac{t^3}{3}+xt \Big)\,dt,
\end{equation}
see e.g. \cite[Chapter 9]{NIST}. 
Note that as in the Hermitian random matrix theory, one needs to use different scaling for the edge case. 

\begin{thm}[\textbf{Edge scaling limits of planar Wishart ensembles at weak non-Hermiticity}] 
\label{Thm_EdgeWeakNonHermiticity}
 Let 
\begin{equation}
\label{EdgeWeakTau1}
\tau =1-\frac{c^2}{(2N)^{1/3}}, \qquad  c \in (0,\infty).
\end{equation}
and $p=(1+\tau)^2$. Then we have the following.
\begin{itemize}
     \item[(i)] \textbf{\textup{(Complex ensemble)}} 
For any $\nu>-1$, we have  
\begin{equation}
\lim_{N \to \infty} R_{N,k}^{ \rm c }(z_1,\dots,z_k) = \det \Big[ K^{c}_{\rm{e}}(z_j,z_l) \Big]_{j,l=1}^k, 
\end{equation}
uniformly on compact subsets of $\C$, where 
\begin{equation}  \label{GinUE wH edge kernel}
\begin{split}
K^{c}_{\rm{e}}(z,w) &=  2c^2\sqrt{2\pi}e^{-(\im z)^2-(\im w)^2}
\\
&\quad \times \int_{-\infty}^{0}e^{\frac{\sqrt{2}}{2}c^3(z+\overline{w}-2u)+\frac{c^6}{6}}
\Ai\Bigl(\sqrt{2}c(z-u)+\frac{c^4}{4} \Bigr)
\Ai\Bigl(\sqrt{2}c(\overline{w}-u)+\frac{c^4}{4} \Bigr)du.
\end{split}
\end{equation} 


\smallskip 
 \item[(ii)] \textbf{\textup{(Symplectic ensemble)}} 
For any $\nu>-1/2$, we have 
\begin{equation}
\lim_{N \to \infty} R_{N,k}^{ \rm s }(z_1,\dots,z_k) = \prod_{j=1}^k (\overline{z}_j-z_j) \, \Pf \bigg[ e^{-|z_j|^2-|z_l|^2} 
\begin{pmatrix} 
\kappa^{ \widetilde{c} }_{ \rm e }(z_j,z_l) & \kappa^{ \widetilde{c} }_{ \rm e }(z_j,\overline{z}_l)
\smallskip 
\\
\kappa^{ \widetilde{c} }_{ \rm e }(\overline{z}_j,z_l) & \kappa^{ \widetilde{c} }_{ \rm e }(\overline{z}_j,\overline{z}_l) 
\end{pmatrix} 
\bigg]_{j,l=1}^k, 
\end{equation}  
uniformly on compact subsets of $\C$, where $ \widetilde{c} = 2^{1/6} c$ and 
\begin{equation}    \label{GinSE wH edge kernel}
\kappa^c_{ \rm e }(z,w):=	 
\sqrt{\pi} e^{z^2+w^2} \int_{-\infty}^{0} W(f_{w,c},f_{z,c})(u) \, du,
\end{equation} 
and
\begin{equation}
f_{z,c}(u):=2c\int_0^u e^{c^3(z-t)+\frac{c^6}{12}}
\Ai\Bigl(2c(z-t)+\frac{c^4}{4}\Bigr)\, dt.
\end{equation}
\end{itemize}

\end{thm}

\medskip 


For the elliptic GinUE, the edge scaling limit at weak non-Hermiticity was first obtained by Bender in \cite{Be10}, see \cite[Section 6.1]{AB23} for a simpler proof. (This is also closely related to the statistics of the rightmost eigenvalue \cite{CESX22,CESX23}.)
However, the form of the limiting kernel in \cite{Be10} was not in the shape of \eqref{GinUE wH edge kernel}, but rather in a certain double contour integral form. 
The limiting kernel of the form \eqref{GinUE wH edge kernel} was obtained by Akemann and Bender in \cite{AB10}, where they studied the non-Hermitian Wishart ensemble in the Dirac picture, i.e. the $(2N+\nu) \times (2N+\nu)$ random Dirac matrix
\begin{equation} \label{def of Dirac}
\mathcal{D}:=\begin{pmatrix} 
0 & X_1 \\
X_2^*  & 0 
\end{pmatrix}. 
\end{equation} 
The eigenvalues of the non-Hermitian Wishart matrix \eqref{def of nWishart} can be obtained by squaring the eigenvalues of the Dirac matrix \eqref{def of Dirac}. 
As a consequence, the origin scaling limits of \eqref{def of nWishart} and \eqref{def of Dirac} are indeed equivalent, see \cite{ABK21}.  
However, away from the origin, it requires separate analysis, and in particular the result in \cite{AB10} does not imply Theorem~\ref{Thm_EdgeWeakNonHermiticity} (i), albeit they are contained in the same universality class.

For the symplectic ensemble, the edge scaling limit \eqref{GinSE wH edge kernel} was obtained in \cite{AP14, BES23} for the elliptic GinSE. Similar to the bulk scaling limit, there is no known example other than the elliptic GinSE where we have the scaling limit \eqref{GinSE wH edge kernel}. Thus, we provide the first result showing the appearance of the universality class \eqref{GinSE wH edge kernel} other than the elliptic GinSE model. 

We mention that both the limiting processes with the kernels \eqref{GinUE wH bulk kernel} and \eqref{GinSE wH bulk kernel} interpolate sine point processes ($c\to 0$) with the bulk Ginibre point processes ($c\to \infty$).
Similarly, the limiting processes with the kernels \eqref{GinUE wH edge kernel} and \eqref{GinSE wH edge kernel} interpolate Airy point processes ($c\to 0$) with the boundary Ginibre point processes ($c\to \infty$). 

\medskip

\begin{rem}[Scaling limits of the elliptic Ginibre ensembles]
As previously mentioned, the non-Hermitian Wishart ensembles are chiral counterparts of the elliptic Ginibre ensembles, which find several applications such as in the equilibrium counting \cite{FK16,BFK21,BEDPMW13}.  
For the reader's convenience, in Table~\ref{Table_summary elliptic Gin} below, let us also provide a summary of the literature on scaling limits of the elliptic Ginibre ensembles we have mostly mentioned above. 
We also refer to \cite{BL22,BL22a} for bulk and edge spacing distributions of the elliptic GinUE and also \cite{AT24} for the equicontinuity of a general $\beta$-ensemble with Hele-Shaw type potentials.    
Note that, for the elliptic Ginibre ensembles, there is no need to distinguish the origin scaling limit since there are no singularities. 
Nonetheless, the origin scaling limit is technically easier than the other cases. 
For instance, \cite{AEP22,Kan02} first obtained the bulk scaling limits at the origin, which were later extended in \cite{BE23,BES23} along the entire real axis.  
Let us also mention \cite{ADM23,Mo23} for recent works on the higher dimensional analogue of the elliptic GinUE, see also \cite{HHJK23,OYZ23,Ch23a} and references therein for more recent work related to the elliptic ensembles.

\begin{table}[h!]
    \centering
   {\def\arraystretch{2}
\begin{tabular}{ |p{1.5cm}|p{1.5cm}|p{1.5cm}|  }
\hline
\multicolumn{3}{|c|}{\textbf{Elliptic GinUE} } 
\\
\hline
&  \centering \textup{Bulk} &  \hspace{0.35cm}\textup{Edge} 
\\ 
\hline 
\centering Strong &  \centering  \cite{Ri13} & \hspace{0.35cm} \cite{LR16}
 \\ 
 \hline
\centering Weak & \centering \cite{FKS97,ACV18} & \hspace{0.35cm} \cite{Be10} 
 \\ 
   \hline 
\end{tabular}
\qquad 
\begin{tabular}{ |p{1.5cm}|p{1.5cm}|p{1.5cm}|  }
\hline
\multicolumn{3}{|c|}{\textbf{Elliptic GinSE} } 
\\
\hline
&  \centering \textup{Bulk} &  \hspace{0.35cm}\textup{Edge} 
\\ 
\hline 
\centering Strong &  \centering  \cite{AEP22,BE23} & \hspace{0.35cm} \cite{BE23}
 \\ 
 \hline
\centering Weak & \centering \cite{Kan02,BES23} & \hspace{0.15cm} \cite{AP14,BES23} 
 \\ 
   \hline 
\end{tabular}
}
\caption{Summary of the previous results on scaling limits of complex/symplectic elliptic Ginibre ensembles.}
\label{Table_summary elliptic Gin}
\end{table}
\end{rem}

\begin{rem}[Real orthogonal ensemble] 
Beyond the complex and symplectic ensembles, there have also been developments of the non-Hermitian ensembles in the real orthogonal symmetry classes.
In this case, both the real and complex eigenvalues appear with non-trivial probabilities, see e.g. \cite{Ta22}. 
The fundamental model in this symmetry class is the real Ginibre ensemble (GinOE), and its integrable structures and scaling limits have been studied in \cite{AK07, FN07, BS09, So07}. 
See also \cite{BF23b, WCF23, TZ23} and references therein for more recent works on the GinOE. The elliptic GinOE has been extensively studied in \cite{AK22, By23b, Fo23, FN08, BMS23, FT21, CFW24}. Its chiral part, the asymmetric Wishart ensemble, has also been investigated in \cite{AKP10, APS10}. 
Indeed, our method of the differential equation (Theorem~\ref{Thm_ODE}) can also be applied to the real orthogonal ensemble, and we hope to revisit this topic in future work. 
\end{rem}

\subsection*{Organisation of the paper}
The rest of this paper is organised as follows. In the next section, we revisit integrable structures of the non-Hermitian Wishart ensemble and compile asymptotic behaviours of the Laguerre polynomials needed for our analysis.
In Section~\ref{Section_ODE}, we establish the differential equation (Theorem~\ref{Thm_ODE}) satisfied by the correlation kernels. 
Subsequent sections implement the asymptotic analysis facilitated by Theorem~\ref{Thm_ODE}. Section~\ref{Section_strong} is devoted to the scaling limits at strong non-Hermiticity (Theorem~\ref{Thm_StrongNonHermiticity}). 
Moving on to Section~\ref{Section_weak bulk}, we derive the scaling limits at weak non-Hermiticity (Theorems~\ref{Thm_BulkWeakNonHermiticity} and ~\ref{Thm_EdgeWeakNonHermiticity}).

\medskip 

\subsection*{Acknowledgements} 
We wish to express our gratitude to Markus Ebke for his helpful suggestion on Figure~\ref{Fig_Omega}. We also thank Gernot Akemann for his interest and helpful discussions.

\section{Preliminaries} \label{Section_preliminaries}

In this section, we compile integrable structures of the non-Hermitian Wishart ensembles and asymptotic behaviours of the generalised Laguerre polynomials.

\subsection{Planar (skew)-orthogonal Laguerre polynomials} \label{Subsec_Laguerre}

In this subsection, we recall the planar (skew)-orthogonal polynomial formalism for determinantal/Pfaffian structures of the non-Hermitian Wishart ensembles. 
Let
\begin{equation}
\label{LaguerrePolynomial}
p_n(z)=(-1)^nn!\Bigl(\frac{\tau}{N}\Bigr)^n L_{n}^{(\nu)}\Big(\frac{N}{\tau}z\Big)
\end{equation}
be the scaled monic Laguerre polynomial. 
It was shown in \cite{Os04,Ak05} that the family $\{p_n\}$ satisfies the planar orthogonality 
\begin{equation}
\int_{\C}p_n(z)\overline{p_m(z)}  \omega^{ \rm c }_N(z) \, dA(z)
=h_{n}\,\delta_{n,m}, \qquad h_{n} =\frac{1-\tau^2}{2}\frac{n! \, \Gamma(n+\nu+1)}{N^{\nu+2n+2}},
\end{equation}
where $\omega^{ \rm c }$ is the weight function in \eqref{weight functions}.
Note that in the maximally non-Hermitian limit $\tau \to 0$, the orthogonal polynomial $p_n$ becomes a monomial. This is consistent with the rotational symmetry of the potential \eqref{weight functions} for $\tau=0$. In this case, some more explicit computations become possible. For instance, the fluctuation of the spectral radius was investigated in \cite{CJQ20}. 

\medskip 

Next, let us recall the skew-orthogonal polynomial formalism due to Kanzieper \cite{Kan02}.
Recall that the skew-symmetric inner product is given by 
\begin{equation}
\langle h|g \rangle_{ \rm s }:= \int  ( \overline{z}-z )  \omega_N^{ \rm s }(z)  \Big( h(z) \overline{g(z)}-\overline{h(z)}g(z) \Big)\, dA(z). 
\end{equation}
Here, one can also consider a general weight function. 
By definition, the family of skew-orthogonal polynomials $\{q_k\}$ is characterized by   
\begin{equation}
\langle q_{2k+1} | q_{2l} \rangle_{ \rm s } =-\langle q_{2l} | q_{2k+1} \rangle_{ \rm s } =r_k \delta_{kl}  , \qquad   \langle q_{2k+1} | q_{2l+1} \rangle_{ \rm s }=\langle q_{2l} | q_{2k} \rangle_{ \rm s }= 0, 
\end{equation}
where $r_k$ is called the skew-norm. 
(Note that the skew-orthogonal polynomials are not uniquely determined by this condition.) 
The construction of skew-orthogonal polynomials for a given weight function remains open in general. Nonetheless, a certain construction has been addressed in a recent work \cite{AEP22} when the associated planar orthogonal polynomial satisfies the three-term recurrence relation. 
(See \cite{Kan02,Ak05} for earlier works on special cases.)
In our present case with the weight function \eqref{weight functions}, it follows from \cite[Example A.3]{AEP22} that the family 
\begin{align}
q_{2k+1}(z)=&-(2k+1)!\Bigl(\frac{\tau}{2N}\Bigr)^{2k+1}L_{2k+1}^{(2\nu)}\Bigl(\frac{2N}{\tau}z\Bigr)
\\
q_{2k}(z)=&\frac{2^{2k}k!\Gamma(k+\nu+1)}{(2N)^{2k}}
\sum_{j=0}^{k}\tau^{2j}\frac{(2j)!}{2^{2j}j!\Gamma(j+\nu+1)} L_{2j}^{(2\nu)}\Bigl(\frac{2N}{\tau}z\Bigr).
\end{align}
form skew-orthogonal polynomials associated with the weight $\omega^{ \rm s }_N(z)$ with the skew-norm 
\begin{equation}
r_k=\frac{(1-\tau^2)^2(2k+1)!\Gamma(2k+2\nu+2)}{(2N)^{4k+2\nu+4}}.
\end{equation}
Let us also mention that the partition functions $Z_N^{ \rm c}$ and $Z_N^{ \rm s }$ in \eqref{Gibbs complex} and \eqref{Gibbs symplectic} can be written in terms of the product of orthogonal and skew-orthogonal norms. 

Using these polynomials together with the general theory on determinantal/Pfaffian point process \cite{BF22,BF23}, we have the following integrable structure of the $k$-point correlation functions \eqref{def of RNk complex} and \eqref{def of RNk symplectic}.

\begin{prop}[\textbf{Determinantal/Pfaffian point processes}] \label{Prop_Det Pfa struc}
For any $\tau \in [0,1]$ and $N \in \mathbb{N}$, we have the following. 
\begin{itemize}
    \item[(i)] \textbf{\textup{(Complex ensemble)}} We have
    \begin{equation} \label{def of KN c}
\bfR_{N,k}^{ \rm c }(z_1,\dots,z_k) = \det \Big[ \bfK_N^{ \rm c }(z_j,z_l) \Big]_{j,l=1}^k, \qquad  \bfK_N^{ \rm c }(z,w):= \sqrt{ \omega^{ \rm c }_N(z) \omega^{ \rm c }_N(w) } S_N(z, \overline{w} ), 
\end{equation}
where 
\begin{equation} \label{def of SN}
S_{N}(z,w) :=\sum_{j=0}^{N-1}\frac{p_j(z) p_j(w) }{ h_j } 
=\frac{2N^{\nu+2}}{1-\tau^2}\sum_{j=0}^{N-1}\frac{j!\, \tau^{2j} }{\Gamma(j+\nu+1)}
L_{j}^{(\nu)}\Bigl(\frac{N}{\tau}z\Bigr) L_{j}^{(\nu)}\Bigl(\frac{N}{\tau}w\Bigr) . 
\end{equation}
\item[(ii)] \textbf{\textup{(Symplectic ensemble)}} We have
\begin{equation}  
\bfR_{N,k}^{ \rm s }(z_1,\cdots, z_k):= \prod_{j=1}^{k} (\overline{z}_j-z_j )  \Pf \Big[ \bfK_{N}^{ \rm s }(z_j,z_l)  \Big]_{ j,l=1 }^k, 
\end{equation} 
where  
\begin{equation}  \label{def of KN s}
\bfK_N^{ \rm s }(z,w) :=  \sqrt{ \omega^{ \rm s }_N(z) \omega^{ \rm s }_N(w) }
\begin{pmatrix} 
\bfkappa_N(z,w) & \bfkappa_N(z,\overline{w}) 
\smallskip 
\\
\bfkappa_N(\overline{z},w) & \bfkappa_N(\overline{z},\overline{w}) 
\end{pmatrix}, \qquad 
\bfkappa_N(z,w):=\bfG_N(z,w)-\bfG_N(w,z).
\end{equation}
Here,  
\begin{align}
\begin{split}
&\quad \bfG_N(z,w)  :=\sum_{k=0}^{N-1}\frac{q_{2k+1}(z)q_{2k}(w)}{r_k} 
\\
&= -\frac{\sqrt{\pi}(2N)^{2\nu+3}}{2^{2\nu+1}(1-\tau^2)^2}
\sum_{k=0}^{N-1}\sum_{j=0}^{k}\frac{(2k)!!\tau^{2k+1}}{2^k\Gamma(k+\nu+3/2)}
\frac{\tau^{2j}(2j-1)!!}{2^j\Gamma(j+\nu+1)}L_{2k+1}^{(2\nu)}\Bigl(\frac{2N}{\tau}z\Bigr)L_{2j}^{(2\nu)}\Bigl(\frac{2N}{\tau}w\Bigr).
\end{split}
\end{align} 
\end{itemize}
\end{prop}

For the proof of Theorem~\ref{Thm_ODE} in the following section, let us here recall some identities of the generalised Laguerre polynomials.
They satisfy the recurrence relations 
\begin{align}
& L_k^{(\nu)}(z)=L_{k}^{(\nu+1)}(z)-L_{k-1}^{(\nu+1)}(z),
\label{Laguerre1}
\\
& kL_{k}^{(\nu)}(z)=(k+\nu)L_{k-1}^{(\nu)}(z)-zL_{k-1}^{(\nu+1)}(z)
\label{Laguerre2}.
\end{align}
Furthermore it satisfies the differentiation rules:
\begin{align}
& \partial_zL_k^{(\nu)}(z)=-L_{k-1}^{(\nu+1)}(z),
\label{DLaguerre1}
\\
& \partial_z\bigl(z^{\nu}L_k^{(\nu)}(z)\bigr)=(k+\nu)z^{\nu-1}L_{k}^{(\nu-1)}(z).
\label{DLaguerre2}
\end{align}
Combining these, it also follows that 
\begin{equation}
 z\partial_z^2L_{k}^{(\nu)}(z)+(\nu+1-z)\partial_zL_k^{(\nu)}(z)+kL_k^{(\nu)}(z)=0.
\label{ODELaguerre}
\end{equation}

\subsection{Strong asymptotics of the generalised Laguerre polynomials}


To describe the asymptotic behaviour of planar orthogonal polynomials, it is convenient use the conformal mappings associated with the droplet \eqref{S droplet}.
Let  
\begin{equation}
\label{ConformalMap1}
\phi(z) :=\Bigl(z+\frac{\tau^2}{z}\Bigr)+2\tau
\end{equation}
be the (shifted) Joukowsky transform $\phi:\C\backslash {\rm clo}(\D) \to \C \backslash S$. 
The inverse map $\psi:\C \backslash S \to \C\backslash {\rm clo}(\D) $ is given by 
\begin{equation}
\label{ConformalMap2}
\psi(z):=\frac{z-2\tau}{2}\Bigl(1+\sqrt{1-\frac{4\tau^2}{(z-2\tau)^2}}\Bigr).
\end{equation}
The Schwarz function $\mathcal{S}:\C\backslash[0,4\tau]\to \C$ is given by 
\begin{equation}
\label{Schwarz1}
\mathcal{S}(z):=\overline{\phi(1/\overline{\psi(z)})}= \frac{(z-2\tau)}{2} \bigg( \tau^2+\frac{1}{\tau^2}+\Bigl(\tau^2-\frac{1}{\tau^2}\Bigr)\sqrt{1-\frac{4\tau^2}{(z-2\tau)^2}} \bigg)+2\tau. 
\end{equation}  
Furthermore, by \cite[Eq.(3.11)]{ABK21}, we have 
\begin{equation}
\label{Schwarz2}
\mathcal{S}(z)=\frac{z}{A^2}\bigl(2C(z)+B\bigr)^2,\quad z\in S^{\mathrm{c}}.
\end{equation}
where $A$ and $B$ are given by \eqref{def of A and B} and 
\begin{equation}
\label{Schwarz3}
C(z)=\frac{A}{4} \frac{1-\tau^2}{\tau} \frac{z-\sqrt{(z-2\tau)^2-4\tau^2}}{z} 
\end{equation}
is the Cauchy transform of the measure $\mu$ in \eqref{MeasMu}, see \cite[Eq.(3.29)]{ABK21}. 
Using the Schwarz function \eqref{Schwarz1}, the domain of \eqref{ConformalMap2} can be extended so that $\psi:\C\backslash[0,4\tau]\to\C\backslash{\rm clo}(B_{\tau}(0))$, where $B_{R}(w)=\{z\in\C:|z-w|<R\}$.  

\medskip 

We now compile the strong asymptotic behaviours of the generalised Laguerre polynomials, which can be found in much literature, see e.g. \cite{Van07}. 
In particular, for the planar orthogonal polynomial, the asymptotic behaviours depend on the regions separated by the limiting skeleton, see Figure~\ref{Fig_PR regimes}.  
In our present case, this skeleton is given by the line segment connecting the two foci of the ellipse \eqref{S droplet}.

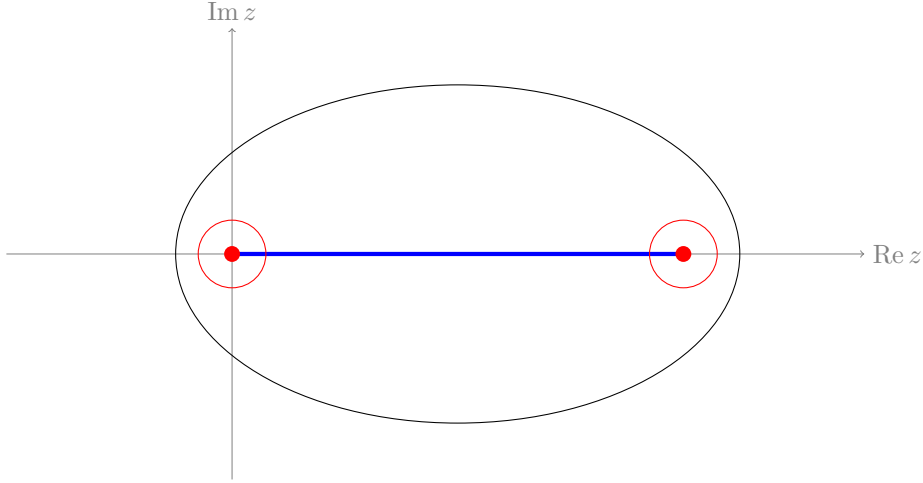
\begin{figure}[t]
    \centering
\begin{tikzpicture}
    \def\t{0.5} 
    \def\scalefactor{3} 
    \begin{scope}[scale=\scalefactor]
       \draw[->,ultra thin,gray] (-1,0)--(2.8,0) node[right]{$\re z$};
\draw[->,ultra thin,gray] (0,-1)--(0,1) node[above]{$\im z$};
    
        \draw[rotate=0] ({2*\t},0) ellipse ({1+\t^2} and {1-\t^2});
        

        
         \draw[ultra thick,blue] ({0},0) -- ({4*\t},0); 
         
         \draw[red] ({4*\t},0) circle (0.15); 
        \draw[red] ({0},0) circle (0.15); 
        \fill[red] ({4*\t}, 0) circle (1pt);
         \fill[red] ({0}, 0) circle (1pt);
    \end{scope}
\end{tikzpicture}
    \caption{In the figure, the closed curve represents the ellipse \eqref{S droplet}, and the red dots indicate its two foci, where $\tau=1/2$. 
    The figure illustrates three different regions to which different forms of the strong asymptotics of the planar Laguerre polynomial $p_n$ in \eqref{LaguerrePolynomial} apply: the critical regime near two foci $\{0,4\tau\}$ (inside the red circles); the oscillatory regime near the line segment connecting two foci (blue full line); and the exponential regime, which covers the rest of the complex plane.}
    \label{Fig_PR regimes}
\end{figure}

The below strong asymptotic of the generalised Laguerre polynomial can be found in \cite[Theorem 2.4 (a)]{Van07}, see also \cite[Appendix B]{LR16} for a similar asymptotic behaviour for the Hermite polynomials. 

\begin{lem}[\textbf{Exponential regime}]
\label{Lem_Strong1}
Fix $\tau\in[0,1)$, $\nu>-1$ and $r \in \mathbb{N}$. 
Then as $N\to\infty$, we have
\begin{equation}
\label{ExpRegimeLaguerre}
L_{N+r}^{(\nu)}\Bigl( \frac{N}{\tau}z \Bigr)=
\frac{1}{\sqrt{2\pi N}}
\frac{(-1)^{N+r}}{\tau^{N+r}}
\frac{\psi(z)^{r+\frac{\nu}{2}}\sqrt{\psi'(z)}}{z^{\nu/2}}e^{Ng_{\tau}(z)}\Bigl(1+O\Bigl(\frac{1}{N}\Bigr) \Bigr),
\end{equation}
uniformly over any compact subset of $\C\backslash[0,4\tau]$, where $g \equiv g_{\tau}$ is defined by 
\begin{equation}\label{GTZ} 
g_{\tau}(z):=\frac{2z}{z+\sqrt{(z-2\tau)^2-4\tau^2}}
+\log\Big( \frac{z-2\tau+\sqrt{(z-2\tau)^2-4\tau^2}}{2} \Big). 
\end{equation}
\end{lem}

We mention that the $g$-function \eqref{GTZ} is indeed a general concept typically encountered within the framework of Riemann-Hilbert analysis. 
This function represents the (complexified) logarithmic energy of the limiting empirical measure of zeros of orthogonal polynomials, which in our present case, is a scaled Marchenko-Pastur law supported on $[0,4\tau]$.

By \cite[Lemma B.1]{AB23} and \cite[Theorem 2.4]{Van07}, we have the following asymptotic behaviour.

\begin{lem}[\textbf{Oscillatory regime}]
\label{StrongBulkFixNu}
For a fixed small $\delta>0$, we define 
\begin{equation*}
B_{\delta}:=\{z\in\C: \delta<\re z<1-\delta,-\delta<\im z<\delta\}.
\end{equation*}
Fix $\tau\in[0,1)$ and $\nu>-1$. Then, for $z\in B_{\delta}$, we have
\begin{align}
\begin{split}
\label{BulkStrongLaguerre}
L_N^{(\nu)}(4Nz)
& = (-1)^N(4Nz)^{-\nu/2}e^{2Nz}\bigl( 2\pi\sqrt{z(1-z)} \bigr)^{-1/2} N^{-1/2} \sqrt{\frac{(N+\nu)!}{N!}}
\\
&\quad \times
\bigg(
\cos\Bigl( 2N\sqrt{z(1-z)}-(2N+\nu+1)\arccos\sqrt{z}+\frac{\pi}{4} \Bigr)\bigl( 1+O(N^{-1}) \bigr)
\\
& \qquad +
\cos\Bigl( 2N\sqrt{z(1-z)}-(2N+\nu-1)\arccos\sqrt{z}+\frac{\pi}{4} \Bigr)O(N^{-1})
\bigg),
\end{split}
\end{align}
as $N\to\infty$.
\end{lem}

By \cite[Theorem 2.4]{Van07} or \cite[Theorem 8.22.8 (c)]{Szego39}, we have the following.
\begin{lem}[\textbf{Critical regime}]
\label{StrongEdgeFixNu}
For a fixed $r\in\mathbb{Z}$, we have that for $x=4N+2(2N)^{1/3}\xi$,
\begin{align}
e^{-\frac{x}{2}}L_{N+r}^{(\nu)}(x)
= (-1)^{N+r}2^{-\nu-1/3}N^{-1/3}
\bigg(\Ai(\xi)-\frac{2r+\nu+1}{(2N)^{1/3}}\Ai'(\xi)+O(N^{-2/3}) \bigg),
\end{align}
as $N\to\infty$, uniformly for $\xi$ in a compact subset of $\C$. 
\end{lem}

\section{Differential equations for correlation kernels} \label{Section_ODE}

In this section, we prove Theorem~\ref{Thm_ODE}.

\subsection{Differential equation of the complex ensemble}

In this subsection, we prove Theorem \ref{Thm_ODE} (i).
\begin{proof}[Proof of Theorem \ref{Thm_ODE} (i)]
Let us make use of the transformation  
\begin{equation}
\widehat{\mathcal{K}}_N^{(\nu)}(z,w):= e^{ \frac{\tau}{1-\tau^2} (z+w) } \mathcal{K}_N^{(\nu)}(z,w),
\end{equation}
where $\mathcal{K}_N^{(\nu)}$ is given by \eqref{def of mathcalK}. 
Then it suffices to show that 
\begin{align}
\begin{split}
&\quad  \bigg[ (1-\tau^2)^2 z\,\partial_z^2
+ (1-\tau^2)^2 (\nu+1)\partial_z - w \bigg] \widehat{\cK}_N(z,w)
\\
&= e^{ \frac{\tau }{1-\tau^2}(z+w)} \frac{N! \, \tau^{2N-1}}{\Gamma(N+\nu+1)}
\bigg[ 
(N+\nu)
L_{N-1}^{(\nu)}\Bigl(\frac{z}{\tau}\Bigr)
L_{N}^{(\nu)}\Bigl(\frac{w}{\tau}\Bigr)
-
\tau^2(N+\nu) \, L_{N}^{(\nu)}\Bigl(\frac{z}{\tau}\Bigr)
L_{N-1}^{(\nu)}\Bigl(\frac{w}{\tau}\Bigr)
 \bigg].
\end{split}
\end{align}

By differentiating $\widehat{\cK}_{N}(z,w)$ with respect to the variable $z$ and using the differentiation rule \eqref{DLaguerre1}, we have 
\begin{align*}
&\quad (1-\tau^2)e^{- \frac{\tau}{1-\tau^2}(z+w)} \partial_{z}\widehat{\cK}_{N}(z,w)
\\
&= -(1-\tau^2)\sum_{ j=1 }^{N-1}\frac{j!\, \tau^{2j-1} }{\Gamma(j+\nu+1)} L_{j-1}^{(\nu+1)}\Bigl(\frac{z}{\tau}\Bigr)L_{j}^{(\nu)}\Bigl(\frac{w}{\tau}\Bigr)   + \sum_{j=0}^{N-1}\frac{j!\,\tau^{2j+1}}{\Gamma(j+\nu+1)} L_{j}^{(\nu)}\Bigl(\frac{z}{\tau}\Bigr)L_{j}^{(\nu)}\Bigl(\frac{w}{\tau}\Bigr). 
\end{align*}
Furthermore, it follows from the recurrence relation \eqref{Laguerre1} that
\begin{align*}
&\quad (1-\tau^2)e^{-\frac{\tau}{1-\tau^2}(z+w)}\partial_{z}\widehat{\cK}_{N}(z,w)- \frac{N!\,\tau^{2N-1}}{\Gamma(N+\nu+1)} L_{N-1}^{(\nu+1)}\Bigl(\frac{z}{\tau}\Bigr)L_{N}^{(\nu)}\Bigl(\frac{w}{\tau}\Bigr)
\\
& = -\sum_{j=0}^{N-1}\frac{(j+1)!\,\tau^{2j+1}}{\Gamma(j+\nu+2)}
L_{j}^{(\nu+1)}\Bigl(\frac{z}{\tau}\Bigr)L_{j+1}^{(\nu)}\Bigl(\frac{w}{\tau}\Bigr)  + \sum_{j=0}^{N-1}\frac{j!\,\tau^{2j+1}}{\Gamma(j+\nu+1)}
L_{j}^{(\nu+1)}\Bigl(\frac{z}{\tau}\Bigr)L_{j}^{(\nu)}\Bigl(\frac{w}{\tau}\Bigr). 
\end{align*}
Note here that by using \eqref{Laguerre2}, 
\begin{align*}
&\quad w \sum_{j=0}^{N-1}\frac{j!\,\tau^{2j}}{\Gamma(j+\nu+2)} L_{j}^{(\nu+1)}\Bigl(\frac{z}{\tau}\Bigr)
L_{j}^{(\nu+1)}\Bigl( \frac{w}{\tau}  \Bigr)
\\
&= -\sum_{j=0}^{N-1}\frac{(j+1)!\,\tau^{2j+1}}{\Gamma(j+\nu+2)}
L_{j}^{(\nu+1)}\Bigl(\frac{z}{\tau}\Bigr)L_{j+1}^{(\nu)}\Bigl(\frac{w}{\tau}\Bigr) + \sum_{j=0}^{N-1}\frac{j!\,\tau^{2j+1}}{\Gamma(j+\nu+1)} L_{j}^{(\nu+1)}\Bigl(\frac{z}{\tau}\Bigr)L_{j}^{(\nu)}\Bigl(\frac{w}{\tau}\Bigr). 
\end{align*}
On the other hand, by \eqref{DLaguerre2}, we have
\begin{align*}
&\quad   e^{-\frac{\tau}{1-\tau^2}(z+w)}  \partial_{z}\bigg[ \frac{N!\,\tau^{2N-1} }{\Gamma(N+\nu+1) } \Bigl(\frac{z}{\tau}\Bigr)^{\nu+1}
L_{N-1}^{(\nu+1)}\Bigl(\frac{z}{\tau}\Bigr)L_{N}^{(\nu)}\Bigl(\frac{w}{\tau}\Bigr) e^{\frac{\tau}{1-\tau^2}(z+w)}
\bigg]
\\
&= \Bigl(\frac{z}{\tau}\Bigr)^{\nu}
\frac{N!}{\Gamma(N+\nu+1)} \frac{\tau^{2N-2}}{1-\tau^2}
\bigg( -N\tau^2L_N^{(\nu)}\Bigl(\frac{z}{\tau}\Bigr)+(N+\nu)
L_{N-1}^{(\nu)}\Bigl(\frac{z}{\tau}\Bigr) \bigg)
L_N^{(\nu)}\Bigl(\frac{w}{\tau}\Bigr).
\end{align*}
Similarly, we obtain 
\begin{align*}
&\quad \frac{1-\tau^2}{w}
\Bigl(\frac{z}{\tau}\Bigr)^{-\nu} e^{-\frac{\tau}{1-\tau^2}(z+w)} \partial_{z}\bigg[
w\Bigl(\frac{z}{\tau}\Bigr)^{\nu+1} \sum_{j=0}^{N-1}\frac{j!\,\tau^{2j}}{\Gamma(j+\nu+2)}
L_{j}^{(\nu+1)}\Bigl(\frac{z}{\tau}\Bigr)
L_{j}^{(\nu+1)}\Bigl(\frac{w}{\tau}\Bigr) e^{\frac{\tau}{1-\tau^2}(z+w)}
\bigg]
\\
&= (1-\tau^2) \sum_{j=0}^{N-1}\frac{j!\,\tau^{2j-1}}{\Gamma(j+\nu+1)}
L_{j}^{(\nu)}\Bigl(\frac{z}{\tau}\Bigr) L_{j}^{(\nu+1)}\Bigl(\frac{w}{\tau}\Bigr)  + \frac{z}{\tau} \sum_{j=0}^{N-1}\frac{j!\,\tau^{2j+1}}{\Gamma(j+\nu+2)} L_{j}^{(\nu+1)}\Bigl(\frac{z}{\tau}\Bigr)
L_{j}^{(\nu+1)}\Bigl(\frac{w}{\tau}\Bigr)
\\
&=  \sum_{j=0}^{N-1}\frac{j!\, \tau^{2j-1}}{\Gamma(j+\nu+1)}
L_{j}^{(\nu)}\Bigl(\frac{z}{\tau}\Bigr)
L_{j}^{(\nu+1)}\Bigl(\frac{w}{\tau}\Bigr)
- \sum_{j=0}^{N-1}\frac{(j+1)!\,\tau^{2j+1}}{\Gamma(j+\nu+2)} 
L_{j+1}^{(\nu)}\Bigl(\frac{z}{\tau}\Bigr) L_{j}^{(\nu+1)}\Bigl(\frac{w}{\tau}\Bigr).
\end{align*}
Combining all of the above, we obtain 
\begin{align*}
&\quad e^{-\frac{\tau}{1-\tau^2}(z+w)} \partial_{z}\Bigl[w \Bigl( \frac{z}{\tau}\Bigr)^{\nu+1} 
 \sum_{j=0}^{N-1}\frac{j!\,\tau^{2j}}{\Gamma(j+\nu+2)}
L_{j}^{(\nu+1)}\Bigl(\frac{z}{\tau}\Bigr)
L_{j}^{(\nu+1)}\Bigl(\frac{w}{\tau}\Bigr)
e^{\frac{\tau}{1-\tau^2}(z+w)}
\Bigr]
\\
&= \frac{w}{\tau(1-\tau^2)}
\Bigl(\frac{z}{\tau}\Bigr)^{\nu}
\bigg[ 
\sum_{j=0}^{N-1}\frac{j!\,\tau^{2j}}{\Gamma(j+\nu+1)}
L_{j}^{(\nu)}\Bigl(\frac{z}{\tau}\Bigr) L_{j}^{(\nu)}\Bigl(\frac{w}{\tau}\Bigr) - \frac{N!\,\tau^{2N}}{\Gamma(N+\nu+1)}
L_{N}^{(\nu)}\Bigl(\frac{z}{\tau}\Bigr)
L_{N-1}^{(\nu+1)}\Bigl(\frac{w}{\tau}\Bigr)
\bigg].
\end{align*}
As a consequence together with \eqref{Laguerre2}, it follows that 
\begin{align*}
&\quad (1-\tau^2)\Bigl( z \partial_{z}^2\widehat{\cK}_N(z,w)
+ (\nu+1)\partial_{z}\widehat{\cK}_N(z,w) \Bigr)
-\frac{w}{1-\tau^2}\widehat{\cK}_N(z,w)
\\
&= e^{\frac{\tau}{1-\tau^2}(z+w)}\frac{N!}{\Gamma(N+\nu+1)}\frac{\tau^{2N-1}}{1-\tau^2}  (N+\nu)
 \bigg[
L_{N-1}^{(\nu)}\Bigl(\frac{z}{\tau}\Bigr)L_{N}^{(\nu)}\Bigl(\frac{w}{\tau}\Bigr) 
-
\tau^2 L_{N}^{(\nu)}\Bigl(\frac{z}{\tau}\Bigr)
L_{N-1}^{(\nu)}\Bigl(\frac{w}{\tau}\Bigr)
\bigg].
\end{align*}
This completes the proof. 
\end{proof}

Note that by change of variables, we have
\begin{align}
\begin{split} 
\label{NonSN}
&\quad\Bigl[
\frac{1-\tau^2}{mN}z\partial_z^2+\Bigl(\frac{(1-\tau^2)(m\nu+1)}{mN}+2\tau z\Bigr)\partial_z+mN\frac{\tau^2z-w}{1-\tau^2}+(m\nu+1)\tau
\Bigr]
S_{mN}^{(m\nu)}(z,w)
\\
&=
\frac{2(mN)^{m\nu+2}}{(1-\tau^2)^2}\frac{(mN)!\,\tau^{2mN-1}}{\Gamma(mN+m\nu)}
\bigg[
L_{mN-1}^{(m\nu)}\Bigl(\frac{mN}{\tau}z\Bigr)L_{mN}^{(m\nu)}\Bigl(\frac{mN}{\tau}w\Bigr)
-\tau^2L_{mN}^{(m\nu)}\Bigl(\frac{mN}{\tau}z\Bigr)L_{mN-1}^{(m\nu)}\Bigl(\frac{mN}{\tau}w\Bigr)
\bigg]. 
\end{split}
\end{align} 
In the latter analysis, we shall use this equation frequently for $m=1,2$. 

\subsection{Differential equation of the symplectic ensemble}

In this subsection, we prove Theorem~\ref{Thm_ODE} (ii). 

\begin{proof}[Proof of Theorem~\ref{Thm_ODE} (ii)]
To lighten notations, we write 
\begin{equation}
f_N(z,w):= -\sum_{k=0}^{N-1}\sum_{j=0}^{k}\frac{(2k)!! \, (2j-1)!! \,\tau^{2k+2j+1} }{2^{k+j}\, \Gamma(k+\nu+3/2) \Gamma(j+\nu+1)} L_{2k+1}^{(2\nu)}\Bigl(\frac{2z}{\tau}\Bigr)L_{2j}^{(2\nu)}\Bigl(\frac{2w}{\tau}\Bigr).
\end{equation}
By differentiating $f_N$ with respect to $z$ and using \eqref{DLaguerre1},
\begin{align*}
 \partial_z f_N(z,w) 
&=  \frac{2}{\tau} \sum_{k=0}^{N-1} \frac{(2k)! \,\tau^{4k+1} }{2^{2k}\, \Gamma(k+\nu+3/2) \Gamma(k+\nu+1)} L_{2k}^{(2\nu+1)}\Bigl(\frac{2z}{\tau}\Bigr)L_{2k}^{(2\nu)}\Bigl(\frac{2w}{\tau}\Bigr)
\\
&\quad + \frac{2}{\tau} \sum_{k=0}^{N-1}\sum_{j=0}^{k-1}\frac{(2k)!! \, (2j-1)!! \,\tau^{2k+2j+1} }{2^{k+j}\, \Gamma(k+\nu+3/2) \Gamma(j+\nu+1)} L_{2k}^{(2\nu+1)}\Bigl(\frac{2z}{\tau}\Bigr)L_{2j}^{(2\nu)}\Bigl(\frac{2w}{\tau}\Bigr).
\end{align*}
Here, by using \eqref{Laguerre1} and \eqref{Laguerre2}, we have 
\begin{align*}
L_{2k}^{(2\nu+1)}\Bigl(\frac{2z}{\tau}\Bigr) &=\frac{1}{2k} \bigg( (2k+1+2\nu) L_{2k-1}^{(2\nu+1)}\Bigl(\frac{2z}{\tau}\Bigr)  -  \frac{2z}{\tau} L_{2k-1}^{(2\nu+2)}\Bigl(\frac{2z}{\tau}\Bigr) \bigg) 
\\
&= \frac{1}{2k} \bigg( (2k+1+2\nu) \Big(   L_{2k-1}^{(2\nu)}\Bigl(\frac{2z}{\tau}\Bigr) + L_{2k-2}^{(2\nu+1)}\Bigl(\frac{2z}{\tau}\Bigr) \Big) -  \frac{2z}{\tau} L_{2k-1}^{(2\nu+2)}\Bigl(\frac{2z}{\tau}\Bigr) \bigg),
\end{align*}
which gives
\begin{align*}
 \partial_{z}f_N(z,w)
&= \frac{2}{\tau}\sum_{k=0}^{N-1}\frac{(2k)!\tau^{4k+1}}{2^{2k}\Gamma(k+\nu+3/2)\Gamma(k+\nu+1)}
L_{2k}^{(2\nu+1)}\Bigl(\frac{2z}{\tau}\Bigr)L_{2k}^{(2\nu)}\Bigl(\frac{2w}{\tau}\Bigr)
\\
&\quad +\frac{2}{\tau}
\sum_{k=1}^{N-1}\sum_{j=0}^{k-1}\frac{(2k-2)!!\,(2j-1)!!\,\tau^{2k+2j+1} }{2^{k+j} \Gamma(k+\nu+3/2)\Gamma(j+\nu+1)}
\\
&\quad \quad \times
\bigg[
(2k+1+2\nu)    \Big(   L_{2k-1}^{(2\nu)}\Bigl(\frac{2z}{\tau}\Bigr) + L_{2k-2}^{(2\nu+1)}\Bigl(\frac{2z}{\tau}\Bigr) \Big)  
- \frac{2z}{\tau}L_{2k-1}^{(2\nu+2)}\Bigl(\frac{2z}{\tau}\Bigr)L_{2j}^{(2\nu)}\Bigl(\frac{2w}{\tau}\Bigr)
\bigg].
\end{align*}

By rearranging the terms and multiplying $(2z/\tau)^{2\nu+1}$, it follows that
\begin{align*}
&\quad \Bigl(\frac{2z}{\tau}\Bigr)^{2\nu+1}\Bigl(\partial_{z}f_N(z,w)+2\tau f_N(z,w)\Bigr)
\\
&= \frac{2}{\tau}\sum_{k=0}^{N-1}\frac{(2k)!\tau^{4k+1}}{2^{2k}\Gamma(k+\nu+3/2)\Gamma(k+\nu+1)} \Bigl(\frac{2z}{\tau} \Bigr)^{2\nu+1}
L_{2k}^{(2\nu+1)}\Bigl(\frac{2z}{\tau} \Bigr)L_{2k}^{(2\nu)}\Bigl(\frac{2w}{\tau}\Bigr)
\\
& \quad +\frac{2}{\tau}\sum_{k=0}^{N-1}\sum_{j=0}^{k}\frac{(2k)!!\,(2j-1)!!\,\tau^{2k+2j+3} }{2^{k+j}\Gamma(k+\nu+3/2) \Gamma(j+\nu+1)}\Bigl(\frac{2z}{\tau} \Bigr)^{2\nu+1}
L_{2k}^{(2\nu+1)}\Bigl(\frac{2z}{\tau}\Bigr)L_{2j}^{(2\nu)}\Bigl(\frac{2w}{\tau}\Bigr)
\\
&\quad -\frac{2}{\tau}\sum_{k=0}^{N-1}\sum_{j=0}^{k}\frac{(2k)!!\,(2j-1)!! \,\tau^{2k+2j+3}}{2^{k+j+1}\Gamma(k+\nu+5/2)\Gamma(j+\nu+1)}\Bigl(\frac{2z}{\tau}\Bigr)^{2\nu+2}
L_{2k+1}^{(2\nu+2)}\Bigl(\frac{2z}{\tau}\Bigr)L_{2j}^{(2\nu)}\Bigl(\frac{2w}{\tau}\Bigr)
\\
&\quad -\frac{2}{\tau}\frac{(2N-2)!!\tau^{2N+1}}{2^{N-1}\Gamma(N+\nu+1/2)}
\Bigl(\frac{2z}{\tau}\Bigr)^{2\nu+1}L_{2N-1}^{(2\nu+1)}\Bigl(\frac{2z}{\tau} \Bigr)\sum_{j=0}^{N-1}\frac{(2j-1)!!\tau^{2j}}{2^j\Gamma(j+\nu+1)}L_{2j}^{(2\nu)}\Bigl(\frac{2w}{\tau}\Bigr)
\\
& \quad +\frac{2}{\tau}\frac{(2N-2)!!\tau^{2N+1}}{2^{N}\Gamma(N+\nu+3/2)}
\Bigl(\frac{2z}{\tau} \Bigr)^{2\nu+2}L_{2N-1}^{(2\nu+2)}\Bigl(\frac{2z}{\tau} \Bigr)\sum_{j=0}^{N-1}\frac{(2j-1)!!\tau^{2j}}{2^j\Gamma(j+\nu+1)}L_{2j}^{(2\nu)}\Bigl(\frac{2w}{\tau} \Bigr).
\end{align*}
Furthermore, by using the Legendre  duplication formula \cite[Eq.(5.5.5)]{NIST}
\begin{equation}\label{Gamma duplication}
\Gamma(2z)= \frac{2^{2z-1}}{ \sqrt{\pi} } \Gamma(z)\Gamma(z+1/2),
\end{equation}
and differentiating the above with respect to $z$, we obtain 
\begin{align*}
&\quad \Bigl(\frac{2z}{\tau}\Bigr)^{-2\nu}
\partial_{z}\bigg[\Bigl(\frac{2z}{\tau}\Bigr)^{2\nu+1}\Bigl(\partial_{z}f_N(z,w)+2\tau f_N(z,w)\Bigr)\bigg]
\\
&=   2\tau z\partial_{z}^2f_N(z,w)+2\tau \Big(2\nu+1-\frac{2z}{\tau}\Big)\partial_{z}f_N(z,w) +\frac{8z}{\tau} f_N(z,w)
\\
& \quad +  \frac{2^{2\nu+3}}{\sqrt{\pi}}
\sum_{k=0}^{N-1}\frac{(2k)!\tau^{4k-1}}{\Gamma(2k+2\nu+1)}
L_{2k}^{(2\nu)}\Bigl(\frac{2z}{\tau}\Bigr)L_{2k}^{(2\nu)}\Bigl(\frac{2w}{\tau}\Bigr)
\\
& \quad - \frac{(2N)!!\tau^{2N-1} }{2^{N-3}\Gamma(N+\nu+1/2)}L_{2N}^{(2\nu)}\Bigl(\frac{2z}{\tau}\Bigr)
\sum_{ j=0 }^{N-1}\frac{(2j-1)!!\tau^{2j}}{2^j\Gamma(j+\nu+1)}
L_{2j}^{(2\nu)}\Bigl(\frac{2w}{\tau}\Bigr), 
\end{align*}
where we used \eqref{Laguerre1}, \eqref{DLaguerre2} and \eqref{ODELaguerre}. 
Therefore, we have shown that $f_N(z,w)$ satisfies the differential equation 
\begin{align}
\begin{split} \label{ODEPart1}
&\quad \bigg[ \frac{1-\tau^2}{2}\,z \, \partial_{z}^2 
+\Big((2\nu+1)\frac{1-\tau^2}{2}+2\tau z\Big)\partial_{z} 
+ \tau(2\nu+1)-2z  \bigg] f_N(z,w) 
\\
&= \frac{2^{2\nu+1}}{\sqrt{\pi}}\sum_{k=0}^{N-1}\frac{(2k)!\tau^{4k} }{\Gamma(2k+2\nu+1)} L_{2k}^{(2\nu)}\Bigl(\frac{2z}{\tau}\Bigr)L_{2k}^{(2\nu)}\Bigl(\frac{2w}{\tau}\Bigr)
\\
&\quad -\frac{(2N)!!\tau^{2N}}{2^{N-1}\Gamma(N+\nu+1/2)} L_{2N}^{(2\nu)}\Bigl(\frac{2z}{\tau}\Bigr)
\sum_{ j=0 }^{N-1}\frac{(2j-1)!!\tau^{2j}}{2^j\Gamma(j+\nu+1)}L_{2j}^{(2\nu)}\Bigl(\frac{2w}{\tau}\Bigr).
\end{split}
\end{align}

Next, we shall derive the differential equation for $f_N(w,z)$. 
As before, by using \eqref{DLaguerre1}, we have 
\begin{align*}
\partial_{z}f_N(w,z)
&= -\frac{2}{\tau}\sum_{k=0}^{N-1}\frac{(2k+1)!\tau^{4k+3}}{2^{2k+1}\Gamma(k+\nu+3/2)\Gamma(k+\nu+2)}
L_{2k+1}^{(2\nu)}\Bigl(\frac{2w}{\tau}\Bigr)L_{2k+1}^{(2\nu+1)}\Bigl(\frac{2z}{\tau}\Bigr)
\\
&\quad + \frac{2}{\tau}\sum_{k=0}^{N-1}\sum_{j=0}^{k}\frac{(2k)!!(2j+1)!!\tau^{2k+2j+3} }{2^{k+j+1} \Gamma(k+\nu+3/2)\Gamma(j+\nu+2)}L_{2k+1}^{(2\nu)}\Bigl(\frac{2w}{\tau}\Bigr)L_{2j+1}^{(2\nu+1)}\Bigl(\frac{2z}{\tau}\Bigr).
\end{align*}
Note that by \eqref{Laguerre2}, we have 
\begin{align*}
(2j+1) L_{2j+1}^{(2\nu+1)}\Bigl(\frac{2z}{\tau}\Bigr)= (2j+2\nu+2)L_{2j}^{(2\nu+1)}\Bigl(\frac{2z}{\tau}\Bigr)-\frac{2z}{\tau}L_{2j}^{(2\nu+2)}\Bigl(\frac{2z}{\tau}\Bigr).
\end{align*}
Using this together with \eqref{Gamma duplication}, it follows that 
\begin{align*}
\partial_{z}f_N(w,z)
&= -\frac{2}{\tau}\sum_{k=0}^{N-1}\frac{(2k+1)!\tau^{4k+3}2^{2\nu+1}}{\sqrt{\pi}\, \Gamma(2k+2\nu+3)}
L_{2k+1}^{(2\nu)}\Bigl(\frac{2w}{\tau}\Bigr)L_{2k+1}^{(2\nu+1)}\Bigl(\frac{2z}{\tau}\Bigr)
\\
&\quad + \frac{2}{\tau}\sum_{k=0}^{N-1}\sum_{j=0}^{k}\frac{(2k)!!(2j-1)!!\tau^{2k+2j+3} }{2^{k+j}\Gamma(k+\nu+3/2)\Gamma(j+\nu+1)}L_{2k+1}^{(2\nu)}\Bigl(\frac{2w}{\tau}\Bigr)L_{2j}^{(2\nu+1)}\Bigl(\frac{2z}{\tau}\Bigr)
\\
& \quad -\frac{2}{\tau}\sum_{k=0}^{N-1}\sum_{j=0}^{k}\frac{(2k)!!(2j-1)!!\tau^{2k+2j+3}}{2^{k+j+1}\Gamma(k+\nu+3/2) \Gamma(j+\nu+2)}L_{2k+1}^{(2\nu)}\Bigl(\frac{2w}{\tau}\Bigr)
\frac{2z}{\tau} L_{2j}^{(2\nu+2)}\Bigl(\frac{2z}{\tau}\Bigr).
\end{align*}
Then by \eqref{Laguerre1}, we have  
\begin{align*}
&\quad \partial_{z}f_N(w,z)+2\tau f_N(w,z)
= -\frac{2}{\tau}\sum_{k=0}^{N-1}\frac{(2k+1)!\tau^{4k+3}2^{2\nu+1}}{\sqrt{\pi}\, \Gamma(2k+2\nu+3)}
L_{2k+1}^{(2\nu)}\Bigl(\frac{2w}{\tau}\Bigr)L_{2k+1}^{(2\nu+1)}\Bigl(\frac{2z}{\tau}\Bigr)
\\
&\quad + \frac{2}{\tau}\sum_{k=0}^{N-1}\sum_{ j=1 }^{k}\frac{(2k)!!(2j-1)!!\tau^{2k+2j+3} }{2^{k+j}\Gamma(k+\nu+3/2)\Gamma(j+\nu+1)}L_{2k+1}^{(2\nu)}\Bigl(\frac{2w}{\tau}\Bigr) L_{2j-1}^{(2\nu+1)}   \Bigl(\frac{2z}{\tau}\Bigr)
\\
& \quad -\frac{2}{\tau}\sum_{k=0}^{N-1}\sum_{j=0}^{k}\frac{(2k)!!(2j-1)!!\tau^{2k+2j+3}}{2^{k+j+1}\Gamma(k+\nu+3/2) \Gamma(j+\nu+2)}L_{2k+1}^{(2\nu)}\Bigl(\frac{2w}{\tau}\Bigr)
\frac{2z}{\tau} L_{2j}^{(2\nu+2)}\Bigl(\frac{2z}{\tau}\Bigr).
\end{align*}
Multiplying $(z/\tau)^{2\nu+1}$ and differentiating the above with respect to $z$ together with \eqref{Laguerre1} and \eqref{DLaguerre2},
\begin{align*}
& \quad \Bigl(\frac{2z}{\tau} \Bigr)^{-2\nu}   \partial_z   \bigg[\Bigl(\frac{2z}{\tau} \Bigr)^{2\nu+1}\Bigl(\partial_{z}f_N(w,z)+2\tau f_N(w,z)\Bigr)\bigg]
\\
&= - \sum_{k=0}^{N-1}\frac{(2k+1)!\tau^{4k+1}2^{2\nu+3}}{\sqrt{\pi} \, \Gamma(2k+2+2\nu)}
L_{2k+1}^{(2\nu)}\Bigl(\frac{2w}{\tau} \Bigr)L_{2k+1}^{(2\nu)}\Bigl(\frac{2z}{\tau} \Bigr)
\\
& \quad - \sum_{k=0}^{N-1}  \sum_{j=0}^{k}   \frac{(2k)!!(2j-1)!!\tau^{2k+2j+1} }{2^{k+j-2}\Gamma(k+\nu+3/2) \Gamma(j+\nu+1)}L_{2k+1}^{(2\nu)}\Bigl(\frac{2w}{\tau}\Bigr)
\frac{2z}{\tau} L_{2j}^{(2\nu+1)}\Bigl(\frac{2z}{\tau}\Bigr)
\\
& \quad + \sum_{k=0}^{N-1}  \sum_{j=1}^{k}   \frac{(2k)!!(2j-1)!!\tau^{2k+2j+1} }{2^{k+j-2}\Gamma(k+\nu+3/2) \Gamma(j+\nu+1)}L_{2k+1}^{(2\nu)}\Bigl(\frac{2w}{\tau}\Bigr)
(2j+2\nu) L_{2j-1}^{(2\nu)}\Bigl(\frac{2z}{\tau}\Bigr). 
\end{align*}
Hence, we have shown that $f_N(w,z)$ satisfies 
\begin{align}
\begin{split}
\label{ODEPart2}
&\quad  \bigg[ \frac{1-\tau^2}{2}\,z \, \partial_{z}^2 
+\Big((2\nu+1)\frac{1-\tau^2}{2}+2\tau z\Big)\partial_{z} 
+ \tau(2\nu+1)-2z  \bigg] f_N(w,z) 
\\
&= -\frac{2^{2\nu+1}}{\sqrt{\pi}}\sum_{k=0}^{N-1}\frac{(2k+1)!\tau^{4k+2}}{\Gamma(2k+1+1+2\nu)}
L_{2k+1}^{(2\nu)}\Bigl(\frac{2w}{\tau}\Bigr)L_{2k+1}^{(2\nu)}\Bigl(\frac{2z}{\tau}\Bigr). 
\end{split}
\end{align}

Finally, we combine \eqref{ODEPart1} with \eqref{ODEPart2} and conclude that 
\begin{align*}
& \quad  \bigg[ \frac{1-\tau^2}{2}\,z \, \partial_{z}^2 
+\Big((2\nu+1)\frac{1-\tau^2}{2}+2\tau z\Big)\partial_{z} 
+ \tau( 2\nu +1)-2z  \bigg] \Big( f_N(z,w) -  f_N(w,z) \Big)
\\
&= \frac{2^{2\nu+1}}{\sqrt{\pi}}
\sum_{k=0}^{2N-1}\frac{k!\tau^{2k}}{\Gamma(k+2\nu+1)}
L_{k}^{(2\nu)}\Bigl(\frac{2z}{\tau}\Bigr)L_{k}^{(2\nu)}\Bigl(\frac{2w}{\tau}\Bigr)
\\
&\quad  -\frac{(2N)!!\tau^{2N} }{2^{N-1}\Gamma(N+\nu+1/2)}  L_{2N}^{(2\nu)}\Bigl(\frac{2z}{\tau}\Bigr)
\sum_{j=0}^{N-1}\frac{(2j-1)!!\tau^{2j}}{2^j\Gamma(j+\nu+1)}L_{2j}^{(2\nu)}\Bigl(\frac{2w}{\tau}\Bigr).
\end{align*}
This completes the proof. 
\end{proof}

\begin{rem}\label{Rem_LSEBO}
By the change of variables, we also have 
\begin{align}
\label{RescaleODE}
\begin{split}
&\quad \bigg[\frac{1-\tau^2}{2N}z\, \partial_{z}^2
+\Big( (2\nu+1)\frac{1-\tau^2}{2N}+2\tau z \Big)\partial_{z}
+\Big(\tau(2\nu+1)-2Nz\Big)
\bigg]  \bfkappa_N(z,w)
\\
& = \frac{(2N)^{2\nu+3}}{(1-\tau^2)^2}
\sum_{k=0}^{2N-1}\frac{k!\tau^{2k}}{\Gamma(k+2\nu+1)}
L_{k}^{(2\nu)}\Bigl(\frac{2N}{\tau}z\Bigr)L_{k}^{(2\nu)}\Bigl(\frac{2N}{\tau}w\Bigr)
\\
& \quad -\frac{\sqrt{\pi}(2N)^{2\nu+3}(2N)!!\tau^{2N}}{2^{N+2\nu}(1-\tau^2)^2\Gamma(N+\nu+1/2)} \,L_{2N}^{(2\nu)}\Bigl(\frac{2N}{\tau}z\Bigr)
\sum_{j=0}^{N-1}\frac{(2j-1)!!\tau^{2j}}{2^j\Gamma(j+\nu+1)}L_{2j}^{(2\nu)}\Bigl(\frac{2N}{\tau}w\Bigr).
\end{split}
\end{align}
\end{rem}

In Theorem~\ref{Thm_ODE} (ii), the second inhomogeneous term in \eqref{ODE for kappaN} also satisfies a second order differential equation.
The proof is similar to that of Theorem~\ref{Thm_ODE} (i), and we leave it to the interested reader.

\begin{prop}[\textbf{Differential equation for the inhomogeneous term}]\label{Prop_RemainderODE}
Let 
\begin{equation} \label{def of gN}
g_N(w):=\sum_{j=0}^{N-1}\frac{(2j-1)!!\tau^{2j}}{2^j\Gamma(j+\nu+1)}L_{2j}^{(2\nu)}\Bigl(\frac{2N}{\tau}w\Bigr),
\end{equation}
Then we have 
\begin{equation}
\bigg[ \frac{1-\tau^2}{2N}w \partial_{w}^2
+\Bigl(\frac{1-\tau^2}{2N}(2\nu+1)+2\tau w\Bigr)\partial_{w} - 2N\Bigl(w-\tau\frac{2\nu+1}{2N}\Bigr) \bigg] g_N(w)= \frac{(2N-1)!!\tau^{2N-1}}{2^{N-1} \Gamma(N+\nu)}
L_{2N-1}^{(2\nu)}\Bigl(\frac{2N}{\tau}w\Bigr).
\end{equation}
\end{prop}

It is indeed a slight abuse of notation as we already used the notation $g \equiv g_\tau$ in \eqref{GTZ}. Nonetheless, since the range of the subscript is clearly different, we will continue to use these notations.

\section{Scaling limits at strong non-Hermiticity}  \label{Section_strong}

In this section, we prove Theorem~\ref{Thm_StrongNonHermiticity}. 
Subsections \ref{Subsec_Thm strong (i) bulk} and \ref{Subsec_Thm strong (ii) bulk} are devoted to the proofs of bulk scaling limits, each of which is for the complex and symplectic ensembles. Similarly, Subsections \ref{Subsec_Thm strong (i) edge} and \ref{Subsec_Thm strong (ii) edge} are devoted to the proofs of edge scaling limits.

\subsection{Proof of Theorem~\ref{Thm_StrongNonHermiticity} (i), the bulk case} \label{Subsec_Thm strong (i) bulk}

By \eqref{RNk rescaled complex} and \eqref{def of KN c}, we need to analyse the correlation kernel $\bfK_N^{ \rm (c) }$. 
As we need to take the weight function into account, we use the transformation 
\begin{equation}
\label{BTSM}
\boldsymbol{S}_{N}(z,w) \equiv \boldsymbol{S}_{N}^{(\nu)}(z,w)
: = \sqrt{\frac{\pi}{2AN}}(zw)^{\frac{\nu}{2}+\frac{1}{4}}
\exp\Bigl(-\frac{N}{1-\tau^2}(2\sqrt{zw}-\tau(z+w))\Bigr)
S_{N}^{(\nu)}(z,w),
\end{equation}
where $S_N \equiv S_{N}^{(\nu)}$ is given by \eqref{def of SN}.
Note that $\bfK_{N}^{(\text{c})}$ is expressed in terms of $\boldsymbol{S}_{N}$ as 
\begin{equation}
\label{bfKc}
\bfK_{N}^{(\text{c})}(z,w)
=
\sqrt{\omega_N^{(\text{c})}(z)\omega_N^{(\text{c})}(w)}
\bigg(
\sqrt{\frac{\pi}{2AN}}(zw)^{\frac{\nu}{2}+\frac{1}{4}}
\exp\Bigl(-\frac{N}{1-\tau^2}(2\sqrt{zw}-\tau(z+w))\Bigr)
\bigg)^{-1} \boldsymbol{S}_{N}(z,w). 
\end{equation}
By the asymptotic expansion \cite[Eq.(10.40.2)]{NIST} of the modified Bessel function, for any fixed $\nu$, the weight function $\omega_N^{\rm c}$ satisfies the asymptotic behaviour  
\begin{equation} \label{omega N asymp}
\omega_N^{\rm c}(z)=|z|^{\nu} \sqrt{\frac{\pi}{2AN|z|}} e^{-AN|z|+BN\re z} \Bigl(1+\frac{4\nu^2-1}{8AN|z|}+O((AN|z|)^{-2}) \Bigr),
\end{equation}
as $N\to \infty,$ uniformly for $0<|z|<\infty$.
Therefore we have 
\begin{equation}  \label{WAB1}
\bfK_{N}^{(\text{c})}(z,w)
= c_N(z,\overline{w})
\Bigl(\frac{|zw|}{z\overline{w}}\Bigr)^{\nu/2}
\frac{e^{-\frac{N}{1-\tau^2}|\sqrt{z}-\sqrt{\overline{w}}|^2}}{(|zw|(z\overline{w}))^{1/4}}  \boldsymbol{S}_{N}(z,w)
\Bigl( 1 + O\Bigl(\frac{1}{N}\Bigr)\Bigr)
\end{equation}
as $N\to\infty$, uniformly for $0<|z|,|w|<\infty$, where 
\[
c_N(z,\overline{w})=e^{i\frac{N\tau}{1-\tau^2}(\im(w)-\im(z))}e^{i\frac{2N}{1-\tau^2}\im(\sqrt{z\overline{w}})}. 
\]
Note that $c_N$ is a cocycle (also known as a gauge transformation) which cancel out when forming a determinant. 

By \eqref{NonSN} and \eqref{BTSM}, we have
\begin{equation}
\label{ODE_BTS}
\bigg[
\frac{1-\tau^2}{2mN\sqrt{zw}}z\partial_z^2+\Bigl(1+\frac{1-\tau^2}{4mN\sqrt{zw}}\Bigr)\partial_z-\frac{(1-\tau^2)((2m\nu)^2-1)}{32mNz\sqrt{zw}}
\bigg]\boldsymbol{S}_{mN}^{(m\nu)}(z,w)
= \mathrm{I}_{mN}^{(m\nu)}(z,w),
\end{equation}
where 
\begin{align}
\begin{split}
\label{cIN}
\mathrm{I}_{mN}^{(m\nu)}(z,w)
&:=
\sqrt{\frac{\pi}{2AmN}}
(zw)^{\frac{m\nu}{2}-\frac{1}{4}}
\exp\Bigl(-\frac{mN}{1-\tau^2}(2\sqrt{zw}-\tau(z+w))\Bigr)
\frac{(mN)^{m\nu+2}}{(1-\tau^2)^2}\frac{(mN)!\cdot \tau^{2mN-1}}{\Gamma(mN+m\nu)}
\\
&\quad \times
\bigg[
L_{mN-1}^{(m\nu)}\Bigl(\frac{mN}{\tau}z\Bigr)L_{mN}^{(m\nu)}\Bigl(\frac{mN}{\tau}w\Bigr)
- \tau^2 L_{mN}^{(m\nu)}\Bigl(\frac{mN}{\tau}z\Bigr)L_{mN-1}^{(m\nu)}\Bigl(\frac{mN}{\tau}w\Bigr)
\bigg]. 
\end{split}
\end{align}
We shall analyse the large-$N$ asymptotic behaviour of the differential equation \eqref{BTSM}. 


The asymptotic behaviours of the generalised Laguerre polynomials in different regimes are given in Lemmas~\ref{Lem_Strong1}, \ref{StrongBulkFixNu}, and \ref{StrongEdgeFixNu}. As a consequence of these lemmas, together with Stirling's formula and the well-known asymptotic behaviour of the Airy function \cite[Eq.(9.7.5)]{NIST}, we have the following estimate for the generalised Laguerre polynomials.
Recall that $B_{R}(w)=\{z\in\C:|z-w|<R\}$. 

\begin{lem}\label{Lem_BDLaguerre}
Fix $\tau\in(0,1)$ and $\nu>-1$. 
For $z$ in a compact subset of $\C\backslash B_{\delta}(0)$ with a small $\delta>0$, we have 
\begin{equation}
\Bigl| e^{-N g_{\tau}(z)}L_{N+r}^{(\nu)}\Bigl(\frac{N}{\tau}z\Bigr) \Bigr|
\leq 
\frac{1}{\tau^{N+r}}\cdot O\Bigl(\frac{1}{N^{1/2}}\Bigr),
\end{equation}
where the $O(N^{-1/2})$-term is uniform for $z\in\C$.
\end{lem}

We define 
\begin{equation}
\label{OmegaF}
\Omega(z) := 2\re\Bigl[\frac{|z|-\tau z}{1-\tau^2}- g_\tau(z)\Bigr],
\end{equation} 
where $g_\tau$ is given by \eqref{GTZ}. 
See Figure~\ref{Fig_Omega} for the graph of $z \mapsto \Omega(z).$
We show the non-negativity of $\Omega$, which is indeed a characteristic property of the $g$-function. 

\begin{lem}
\label{Lem_OPOV}
For $z\in \C$, we have $\Omega(z)\geq 0$, and the equality holds only when $z\in \partial S$.
\end{lem}

\begin{proof}
We write 
\begin{equation}
\label{W Omega}
\widetilde{\Omega}(u):=\frac{1}{2}\Omega(\phi(u)),
\end{equation}
where $\phi$ is the conformal map given in \eqref{ConformalMap1}. 
Then it is enough to show
\begin{equation}
\{u\in\C\backslash B_{\tau}(0)\,|\,\widetilde{\Omega}(u)\leq 0\}=\partial\D.
\end{equation}
Note that 
\begin{equation}
\label{Polar g tau}
g_{\tau}(\phi(u))=\log u+\frac{\tau}{u}+1. 
\end{equation}
By using the polar coordinate $u=Re^{i\theta}$ for $R>0$ and $\theta\in[0,2\pi)$, we have 
\begin{align} \label{W Omega v2}
\widetilde{\Omega}(Re^{i\theta}) = \frac{1}{1-\tau^2}
\Bigl(
\frac{R^2+\tau^2+2 R\tau \cos \theta}{ R } -\tau\Bigl(R+\frac{\tau^2}{R}\Bigr)\cos\theta-2\tau^2
\Bigr) - \log R- \frac{\tau}{R}\cos\theta-1,
\end{align}
where we have used $R^2+\tau^2+2R\tau\cos\theta\geq(R-\tau)^2\geq 0$ for all $\theta\in[0,2\pi)$ and $\tau\in(0,1)$. 
Here, note that 
\begin{equation}
\label{W Omega 0}
    \widetilde{\Omega}(e^{i\theta})=0, \qquad \theta\in[0,2\pi). 
\end{equation}
By differentiating \eqref{W Omega v2} with respect to $\theta$ and $R$, 
we have 
\begin{align}
\label{W Omega theta}
\partial_{\theta}\widetilde{\Omega}(Re^{i\theta})
&=
\frac{\tau (R-1)^2 \sin\theta}{(1-\tau^2)R},
\\
\label{W Omega R}
\partial_{R}\widetilde{\Omega}(Re^{i\theta})& =
\frac{(R-1)(R+\tau^2-\tau(R+1)\cos\theta)}{(1-\tau^2)R}. 
\end{align}
Note that the right-hand side of \eqref{W Omega theta} vanishes only when $R=1$ or $\sin\theta=0$. 
When $R=1$, \eqref{W Omega R} also vanishes. 
Since $R+\tau^2-\tau(R+1)\cos\theta\geq(R-\tau)(1-\tau)$ for all $\theta\in[0,2\pi)$, 
when $\sin\theta=0$, \eqref{W Omega R} vanishes only when $R=\tau$. 
Note also that for $\tau \in (0,1)$, 
$$
\widetilde{\Omega}(\tau e^{i\theta}) = -\frac{(1-\tau)(1+\cos\theta)}{1+\tau}-\log\tau \ge -\frac{2(1-\tau)}{1+\tau}-\log\tau > 0. 
$$
This completes the proof. 
\end{proof}

\begin{figure}[t]
    \centering
    \includegraphics[width=0.5\textwidth]{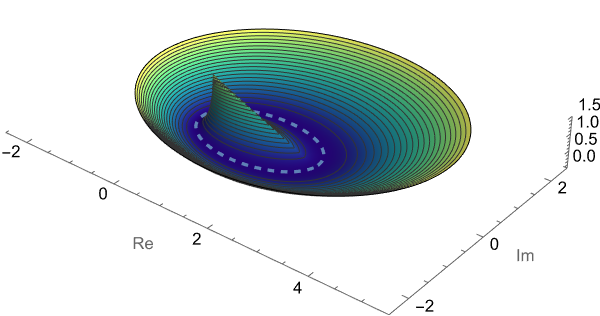}
    \caption{Graph of the function $z \mapsto \Omega(z)$. Here the dashed line indicates the ellipse \eqref{S droplet}.}
    \label{Fig_Omega}
\end{figure}

We now consider the diagonal value $\bfK_N^{ \rm c }(z,\overline{z})$. 
Recall that the droplet $S$ is given by \eqref{S droplet}. 
In the following proposition, we provide the exponential convergence rate of the equilibrium measure \eqref{MeasMu}, i.e. the rate of convergence of \cite[Theorem 1]{ABK21}.   
Note that for an $N$-independent weight function, the asymptotic expansion of the density is known in the literature, see e.g. \cite[Eq.(1.3)]{Am13}.

\begin{prop}\label{Lem_OPDens}
Let $V_{S}$ be a compact subset of $\mathrm{int}(S)$ and $U_{S}$ a neighbourhood of $S$. 
Then, there exists $\epsilon>0$ such that as $N\to \infty$, 
\begin{equation}
\boldsymbol{S}_N^{(\nu)}(z,\overline{z})
=
\begin{cases}
\dfrac{N}{2(1-\tau^2)}+O(e^{-\epsilon N}), & z\in V_{S}, 
\smallskip 
\\
O(e^{-\epsilon N}) & z \not\in U_{S},
\end{cases}
\end{equation}
uniformly over all $z$.
\end{prop}

\begin{proof} 
We fix $z\in V_{S}\subset \mathrm{int}(S)\backslash B_{\delta}(0)$ with a small $\delta>0$. 
Then for a sufficiently large $N$, it follows from Lemma~\ref{Lem_BDLaguerre} and \eqref{cIN} that 
\begin{align}
\begin{split}
\label{BEST1}
\bigl|\mathrm{I}_N^{(\nu)}(z,\overline{z})\bigr|
\leq 
C N \tau^{2N-1}
\Bigl| 
e^{-\frac{N}{1-\tau^2}(2\sqrt{z\overline{z}}-\tau(z+\overline{z}) )}
L_{N-1}^{(\nu)}\Bigl(\frac{N}{\tau}z \Bigr)L_{N}^{(\nu)}\Bigl(\frac{N}{\tau}\overline{z} \Bigr)
\Bigr|
\leq 
C e^{-N\Omega(z)}, 
\end{split}
\end{align}
where $C$ does not depend on $N$, but may change in each line. 
Since $\Omega$ is strictly positive over a compact subset of $\mathrm{int}(S)$ by Lemma~\ref{Lem_OPOV}, the right hand side of this inequality is uniformly bounded by $O(e^{-\epsilon N})$ with $\min_{z}\Omega(z)>\epsilon>0$
over a compact subset of $\mathrm{int}(S)$. 
Therefore, by \eqref{ODE_BTS}, we obtain 
\[
\boldsymbol{S}_N^{(\nu)}(z,\overline{z})=\mathfrak{c}_1(N)+O(e^{-\epsilon N}), \qquad z \in V_S.
\]
Note that $\mathfrak{c}_1(N)$ is an integrating constant which does not depend on $z$. 
By \cite[Theorem 1]{ABK21}, we have that for a bounded continuous function $f$, 
\[
\int f(z) \frac{1}{N}\bfK_N^{ \rm c }(z,\overline{z})dA(z)\to \int f(z) d\mu(z),\quad \text{as $\to\infty$}, 
\]
where $d\mu(z)$ is defined by \eqref{MeasMu}. 
Together with \eqref{WAB1}, we obtain that $\mathfrak{c}_1(N)=N^{}/(2(1-\tau^2))$. 

By Lemma~\ref{Lem_OPOV}, we have the same estimate \eqref{BEST1} over the complement of a neighbourhood $U_{S}$ of $S$. 
By integrating from a fixed point $z_0$ in the complement of $U_{S}$ and \eqref{ODE_BTS}, we have 
\[
\boldsymbol{S}_N^{(\nu)}(z,\overline{z})=\mathfrak{c}_2(N)+O(e^{-\epsilon N}),
\]
where $O(e^{-\epsilon N})$-term is uniform over the complement of $U_{S}$. 
The complement of $U_{S}$ is unbounded, but since $\Omega(z)$ in \eqref{OmegaF} grows linearly as $|z|\to\infty$, the integration of $\frac{1}{N}\bfK_N^{ \rm c }(z,\overline{z})$ converges. 
Therefore, it follows that $\mathfrak{c}_2(N)$ is exponentially small. 
This completes the proof. 
\end{proof}

Next, let us consider the off-diagonal asymptotic behaviour of the kernel $\bfK_N^{ \rm c }$. 

\begin{prop}\label{Prop_Bulk}
Let $V_{S}$ be a compact subset of $\mathrm{int}(S)$ and $U_{S}$ a neighbourhood of $S$. 
Then, as $N \to \infty$, there exist $\epsilon>0$ and $\delta>0$ such that the following asymptotic expansion holds uniformly over all $z,w$ in the complement of $B_{\delta'}(0)$ with a small $\delta'>0$. 
\begin{itemize}
    \item If $w\in V_{S}$, $z\in B_{\delta}(w)$,
\begin{equation}  \label{KN asymptotic 1}
\bfK_N^{ \rm c }(z,\overline{w})= \frac{c_N(z,\overline{w})N}{2(|zw|(z\overline{w}))^{1/4}(1-\tau^2)} \Bigl(\frac{|zw|}{z\overline{w}}\Bigr)^{\nu/2}e^{-\frac{N}{1-\tau^2}|\sqrt{z}-\sqrt{\overline{w}}|^2}
\Bigl(1+O(\frac{1}{N}) \Bigr). 
\end{equation}
\item If $w\in V_{S}$, $z\notin B_{\delta}(w)$ or if $w\notin U_{S}$, $z\in \C$, 
\begin{equation}  \label{KN asymptotic 2}
\bfK_N^{ \rm c }(z,\overline{w})=  O(e^{-\epsilon N}) . 
\end{equation}
\end{itemize}
\end{prop}

We mention that the Szeg\H o type asymptotic behaviours of the correlation kernel when $z$ and $w$ are sufficiently away from each others have also been studied for various random normal matrix models \cite{AC23,BY23,ADM23,FJ96}.

\begin{proof}[Proof of Proposition~\ref{Prop_Bulk}]
We write
\[
\boldsymbol{S}_N^{(\nu)}(z,\overline{w})
=
\boldsymbol{S}_N^{(\nu)}(w,\overline{w})+\int_{w}^z\partial_{\xi}\boldsymbol{S}_N^{(\nu)}(\xi,\overline{w}) \, d\xi. 
\]
For a sufficiently large $N$, by \eqref{OmegaF} and Lemma~\ref{Lem_BDLaguerre}, we have 
\begin{equation}
\label{EST2}
\bigl|\mathrm{I}_N^{(\nu)}(\xi,\overline{w})\bigr|
\leq  
C \exp\Big( -\frac{N}{2}\bigl(\Omega(\xi)+\Omega(w)-\frac{2}{1-\tau^2}|\xi-w| \bigr) \Big),
\end{equation}
for some $C>0$. 
We take $\epsilon>0$ such that $0<2\epsilon<\min_{w\in V_{S}}\Omega(w)$ and $\xi=w\in V_{S}$. 
Then there exists a small $\delta$ such that when $|\xi-w|<\delta<(1-\tau^2)\epsilon$, 
we have 
$$\Omega(\xi)+\Omega(w)-\frac{2}{1-\tau^2}|\xi-w|>2\epsilon>0,\qquad w \in V_S.$$ 
This implies that $|\mathrm{I}_N^{(\nu)}(\xi,\overline{w})|\leq O(e^{-\epsilon N})$ uniformly over $w\in V_{S}$. 
Therefore, by \eqref{WAB1} and \eqref{ODE_BTS}, we have
\[
\boldsymbol{S}_N^{(\nu)}(z,\overline{w})=\frac{N}{2(1-\tau^2)}+O(e^{-\epsilon N}).
\]
Then, by combining with \eqref{WAB1}, we obtain \eqref{KN asymptotic 1}. 

For $z\notin B_{\delta}(w)$, by the same estimate \eqref{EST2}, for any $\xi\in\C$, we have 
\[
\bigl|\partial_{\xi}\boldsymbol{S}_N^{(\nu)}(\xi,\overline{w}) \bigr|e^{-\frac{N}{1-\tau^2}|\sqrt{\xi}-\sqrt{w}|^2}
=
O(e^{-\epsilon N} e^{-\frac{N}{2}\Omega(\xi)}).
\]
By integrating $\partial_{\xi}S_N^{(\nu)}(\xi,\overline{w})$ along the straight line from $w$ to $z$, it follows from Proposition~\ref{Lem_OPDens} that for $w\in V_{S}$ and $z\notin B_{\delta}(w)$,  
\[
\boldsymbol{S}_N^{(\nu)}(z,\overline{w})e^{-\frac{N}{1-\tau^2}|\sqrt{z}-\sqrt{w}|^2}
=
\boldsymbol{S}_N^{(\nu)}(w,\overline{w})e^{-\frac{N}{1-\tau^2}|\sqrt{z}-\sqrt{w}|^2}+O(e^{-\epsilon N})
=
\frac{Ne^{-\frac{N}{1-\tau^2}|\sqrt{z}-\sqrt{w}|^2}}{2(1-\tau^2)}+O(e^{-\epsilon N}). 
\]
Since $|z-w|>\delta>(1-\tau^2)\epsilon$ and $|\sqrt{z}-\sqrt{w}|^2\leq |z-w|$, we have $\bfK_N^{ \rm c }(z,\overline{w})=O(e^{-\epsilon N})$. 
Note here that $z$ can be in the unbounded subset of $\C$, but the error bound is uniform due the linear growth of $\Omega(\xi)$ as $|\xi|\to\infty$. 
This proves the second case. 
When $w\notin U_{S}$, by taking $\epsilon>0$ such that $0<2\epsilon <\min_{w\notin U_{S}}\Omega(w)$, it follows from the same strategy with the second case. 
\end{proof}

We are now ready to complete the proof of \eqref{GinUE bulk kernel} in Theorem~\ref{Thm_StrongNonHermiticity}. 
\begin{proof}[Proof of \eqref{GinUE bulk kernel} in Theorem~\ref{Thm_StrongNonHermiticity}]

We write 
\begin{equation} 
\label{LightNotationZ}
z \equiv z(\zeta):=  p + \mathbf{n}(p)\frac{\zeta}{\sqrt{N\delta}} , \qquad 
w \equiv w(\eta):=  p + \mathbf{n}(p)\frac{\eta}{\sqrt{N\delta}},
\end{equation}
for the local coordinates, see \eqref{def of delta Lap Q}, \eqref{RNk rescaled complex}, and \eqref{RNk rescaled symplectic}.
Note that $\mathbf{n}(p)=1$ for the bulk case. 
To lighten notations, we often write $z$ and $w$ instead of $z(\zeta)$ and $w(\eta)$.
By the Taylor expansion, we have 
\[
-\frac{N}{1-\tau^2}|\sqrt{z}-\sqrt{w}|^2=-\frac{1}{2}|\zeta-\eta|^2+O\Bigl(\frac{1}{\sqrt{N}}\Bigr). 
\]
Then, by Proposition~\ref{Prop_Bulk}, there exist $\epsilon>0$ and a cocycle factor $c_N(\zeta,\eta)$ such that
\[
\frac{1}{N\delta}\bfK_N^{ \rm c }(z,\overline{w})=c_N(\zeta,\eta) e^{-\frac{1}{2}(|\zeta|+|\eta|^2)+\zeta\overline{\eta}}+O(e^{-\epsilon N}),\quad\text{as $N\to\infty$},
\]
uniformly for $\zeta,\eta$ in compact subsets of $\C$. 
Then by Proposition~\ref{Prop_Det Pfa struc} (i) and \eqref{RNk rescaled complex}, the proof of \eqref{GinUE bulk kernel} in Theorem~\ref{Thm_StrongNonHermiticity} is now complete. 
\end{proof}

\subsection{Proof of Theorem~\ref{Thm_StrongNonHermiticity} (ii), the bulk case} \label{Subsec_Thm strong (ii) bulk}

In this subsection, we shall prove \eqref{GinSE bulk kernel} in Theorem~\ref{Thm_StrongNonHermiticity}. 
Recall that $\bfkappa_N$ is given by \eqref{def of KN s}. 
For this purpose, we use the transformation  
\begin{equation}
\label{WKA1}  
\boldsymbol{\kappa}_N(z,w) 
:= \sqrt{\frac{\pi}{4AN}}(zw)^{\nu+\frac{1}{4}}\exp\Bigl( -\frac{2N}{1-\tau^2}(2\sqrt{zw}-\tau(z+w)) \Bigr)
\bfkappa_N(z,w). 
\end{equation}
Then, by Theorem~\ref{Thm_ODE} (ii), we have
\begin{align}
\begin{split}
\label{WKA2}
&\quad
\bigg[
\frac{1-\tau^2}{4N}\sqrt{\frac{z}{w}}\partial_z^2+\Bigl(1+\frac{1-\tau^2}{8N\sqrt{zw}}\Bigr)\partial_z
-\Bigl( \frac{N(z-w)}{\sqrt{zw}(1-\tau^2)}+\frac{(16\nu^2-1)(1-\tau^2)}{2z\sqrt{zw}}\Bigr)
\bigg]
\boldsymbol{\kappa}_N(z,w)
\\
&=\frac{N}{2\sqrt{zw}(1-\tau^2)}\boldsymbol{S}_{2N}^{(2\nu)}(z,w)
-
\frac{1}{2(zw)^{3/4}} \exp\Bigl( \frac{2N}{1-\tau^2}(\sqrt{z^2}+\sqrt{w^2}-2\sqrt{zw}) \Bigr) \boldsymbol{E}_N(z,w),
\end{split}
\end{align}
where $\boldsymbol{S}_N^{(\nu)}$ is given by \eqref{BTSM} and 
\begin{align}
\begin{split}
\label{WEN}
\boldsymbol{E}_N(z,w)
&:=
\sqrt{\frac{\pi}{4AN}}(zw)^{\nu+\frac{1}{2}}
\exp\Bigl( -\frac{2N}{1-\tau^2}\bigl(\sqrt{z^2}+\sqrt{w^2}-\tau (z + w)\bigr) \Bigr)
\\
&\quad \times 
\frac{\sqrt{\pi}(2N)^{2\nu+3}(2N)!!\tau^{2N}}{2^{N+2\nu}(1-\tau^2)^2\Gamma(N+\nu+1/2)}L_{2N}^{(2\nu)}\Bigl(\frac{2N}{\tau}z \Bigr)
g_N(w).
\end{split}
\end{align}
Here, $g_N$ is given by \eqref{def of gN}. 
On the other hand, by Proposition~\ref{Prop_RemainderODE}, we have
\begin{equation}
\label{ODE_WEN}
\bigg[
\frac{1-\tau^2}{4N}\frac{w}{\sqrt{w^2}}\partial_w^2 + \partial_w -\frac{(1-\tau^2)(4\nu^2-1)}{16N w\sqrt{w^2}} 
\bigg]
\boldsymbol{E}_N(z,w)
=
\mathrm{II}_{N}^{(\nu)}(z,w),
\end{equation}
where  
\begin{align}
\begin{split}
\label{cTN}
\mathrm{II}_{N}^{(\nu)}(z,w)
&:=
\sqrt{\frac{\pi}{4AN}}
\frac{(zw)^{\nu+1/2}}{2\sqrt{w^2}} \frac{(2N)^{2\nu+3}(2N)!\tau^{4N-1}}{(1-\tau^2)^2\Gamma(2N+2\nu)}
\\
&\quad \times
\exp\Bigl( -\frac{2N}{1-\tau^2}(\sqrt{z^2}-\tau z) \Bigr)
L_{2N}^{(2\nu)}\Bigl(\frac{2N}{\tau}z \Bigr)
\exp\Bigl( -\frac{2N}{1-\tau^2}(\sqrt{w^2}-\tau w) \Bigr)
L_{2N-1}^{(2\nu)}\Bigl(\frac{2N}{\tau}w \Bigr).
\end{split}
\end{align}


\begin{lem}\label{Lem_BulkSymp}
Let $p\in(e_-,e_+)\backslash\{0\}$. 
Then there exist a neighbourhood of $U\subset\mathrm{int}(S)$ of $p$ and 
an $\epsilon>0$ such that the following estimate holds for $\xi,w\in U$: 
\[
\boldsymbol{S}_{2N}^{(2\nu)}(\xi,w)
=\frac{N}{1-\tau^2}+O(e^{-\epsilon N}),\qquad
\boldsymbol{E}_N(z,w)=O(e^{-\epsilon N}),\quad \text{as $N\to\infty$}. 
\]

\end{lem} 
\begin{proof}
The proof for $\boldsymbol{S}_{2N}^{(2\nu)}$ is same as that of Proposition~\ref{Prop_Bulk} and we skip the details.
For $p\in\mathrm{int}(S)\backslash B_{\delta'}(0)$ with a small $\delta'>0$, by Lemma~\ref{Lem_OPOV}, there exists a compact subset $V_{S}\subset\mathrm{int}(S)\backslash B_{\delta'}(0)$ that contains an open neighbourhood of $p$ and $\Omega(z)>\epsilon$ for all $z\in V_{S}$ with some $\epsilon>0$. 
We take $$U_1\subset V_{S}\cap\{z\in\C\backslash B_{\delta'}(0):|p-z|<(1-\tau^2)\epsilon/8\}.$$ 
Since $\mathrm{II}_N^{(\nu)}(\xi,w)=O(e^{-\epsilon N})$ for $\xi,t\in U_1$, it follows from \eqref{ODE_WEN} and \eqref{cTN} that 
\[
\partial_t\boldsymbol{E}_N(\xi,t)=O(e^{-\epsilon N}),\quad \text{for $\xi,t\in U_1$}. 
\] 
Note that by \eqref{WEN} and the assumption $\nu>-1/2$, we have $\boldsymbol{E}_N(\xi,0)=0$ for $\xi\in U_1$. 
Thus it follows that for $\xi,w\in U_1$,
\[
\boldsymbol{E}_N(\xi,w)
=\boldsymbol{E}_N(\xi,0)
+
\int_0^w\partial_t\boldsymbol{E}_N(\xi,t)dt=O(e^{-\epsilon N}). 
\]
This completes the proof. 
\end{proof}

We are ready to complete the proof of \eqref{GinSE bulk kernel} in Theorem~\ref{Thm_StrongNonHermiticity}.

\begin{proof}[Proof of \eqref{GinSE bulk kernel} in Theorem~\ref{Thm_StrongNonHermiticity}] 
We set 
\begin{equation}
\kappa_N^{(\rm b)}(\zeta,\eta):=\frac{1}{|p|(N\delta)^{3/2}}\boldsymbol{\kappa}_N(z,w),
\end{equation}
where we have used the local coordinates \eqref{LightNotationZ}. 
By  Proposition~\ref{Prop_Det Pfa struc} (ii), \eqref{WKA1}, \eqref{RNk rescaled symplectic}, together with the asymptotic behaviour of the weight function \eqref{WAB1}, it suffices to show that $\kappa_N^{ (\rm b) } \to \kappa_{\rm b}$ as $N \to \infty$, uniformly on compact subsets of $\C$. 

By \eqref{WKA2} and Lemma~\ref{Lem_BulkSymp}, there exists a positive constant $\epsilon>0$ such that
\begin{equation}
\partial_{\zeta}\kappa_N^{(\rm b)}(\zeta,\eta)
=2(\zeta-\eta)\kappa_N^{(\rm b)}(\zeta,\eta)
+2+O(e^{-\epsilon N}), 
\end{equation}
uniformly for $\zeta,\eta$ in compact subsets of $\C$. 
By taking $N \to \infty,$  we arrive at the differential equation for the limiting pre-kernel that has been utilized in \cite{Kan02,BE23}, see also \cite[Propositoin 2.2 (a)]{BF23a}. To be more precise, as stated in \cite[Theorem 1.2 and Proposition 3.1]{BE23}, we can employ the initial condition $\kappa_N^{(\rm b)}(\zeta,\zeta)=0$ to solve the above differential equation, leading to \eqref{GinSE edge kernel}. This completes the proof.
\end{proof}

\subsection{Proof of Theorem~\ref{Thm_StrongNonHermiticity} (i), the edge case} \label{Subsec_Thm strong (i) edge}

As before, we need proper transformations but for the edge scaling limits, we use slightly different conventions. 
We first define 
\begin{align}
\begin{split}
\label{RescaleKernelS}
\widehat{S}_{N}^{(\nu)}(\zeta,\eta)
:= \sqrt{\frac{\pi}{2AN}}
\frac{1}{N\delta} (zw)^{\frac{\nu}{2}-\frac{1}{4}}
\exp\Bigl(-N\frac{2\sqrt{zw}-\tau(z+w)}{1-\tau^2}\Bigr)S_{N}^{(\nu)}(z,w),
\end{split}
\end{align}
where $S_N \equiv S_{N}^{(\nu)}$ is given by \eqref{def of SN} and $\delta$ is given by \eqref{def of delta Lap Q}. 
Then by Theorem~\ref{Thm_ODE} (i), we have
\begin{equation} 
\label{RescaledODES}
\widehat{\mathfrak{D}}_{N,m}[\zeta] \, \widehat{S}_{mN}^{(m\nu)}(\zeta,\eta)
=\widehat{\mathrm{I}}_{mN}^{(m\nu)}(\zeta,\eta),
\end{equation}
where 
\begin{equation}
\widehat{\mathfrak{D}}_{N,m}[\zeta] =  \frac{1-\tau^2}{2mN}\sqrt{\frac{z}{w}}\frac{\sqrt{N\delta}}{\mathbf{n}(p)} \partial_\zeta^2
+
\Bigl(1+\frac{3(1-\tau^2)}{4mN\sqrt{zw}} \Bigr) \partial_\zeta +
\frac{\mathbf{n}(p)}{2z\sqrt{N\delta}}
\Bigl(1-\frac{(4(m\nu)^2-1)(1-\tau^2)}{16mN\sqrt{zw}}\Bigr)
\end{equation}
and 
\begin{align}
    \begin{split}
    \label{WhatIN}
       \widehat{\mathrm{I}}_{mN}^{(m\nu)}(\zeta,\eta)
       & :=\sqrt{\frac{\pi}{2AmN}}
\frac{(mN)^{m\nu+2}\tau^{2mN-1} }{(1-\tau^2)^2(N\delta)^{3/2}} \frac{(mN)!}{\Gamma(mN+m\nu)}
\frac{(zw)^{\frac{m\nu}{2}+\frac{1}{4}}}{zw}
\exp\Bigl(-mN\frac{2\sqrt{zw}-\tau(z+w)}{1-\tau^2}\Bigr)
\\
&\quad\times
\mathbf{n}(p)
\bigg[
L_{mN-1}^{(m\nu)}\Bigl(\frac{mN}{\tau}z\Bigr)L_{mN}^{(m\nu)}\Bigl(\frac{mN}{\tau}w \Bigr)
-
\tau^2 L_{mN}^{(m\nu)}\Bigl(\frac{mN}{\tau}z \Bigr)L_{mN-1}^{(m\nu)}\Bigl(\frac{mN}{\tau}w \Bigr)
\bigg]. 
    \end{split}
\end{align}

\begin{lem}
We define
\begin{equation}
\label{fzw}
f(z,\overline{w})
:=
\frac{2\sqrt{z\overline{w}}-\tau (z+\overline{w})}{1-\tau^2}
- g_{\tau}(z)- g_{\tau}(\overline{w}),
\end{equation}
where $g_\tau$ is given by \eqref{GTZ}. 
Let $z=z(\zeta)$ and $w=w(\eta)$ be the local coordinates given by \eqref{LightNotationZ} for a fixed $p \in \partial S$. 
Then as $N\to\infty$, we have   
\begin{equation}
\label{EFNA}
f(z,\overline{w})
=
\frac{1}{2N}(\zeta+\overline{\eta})^2
+
O\Bigl(\frac{1}{N^{3/2}}\Bigr). 
\end{equation}
\end{lem}

\begin{proof}
Under the local coordinates given by \eqref{LightNotationZ} and by Taylor expansion, we have 
\begin{align}
\begin{split}
\label{Taylor fzw}
f(z,\overline{w})
&=
f(p,\overline{p})
+
\frac{1}{\sqrt{N\delta}}
\Bigl(
\mathbf{n}(p)\partial_{z}f(p,\overline{p})\zeta
+
\overline{\mathbf{n}(p)}\partial_{\overline{w}}f(p,\overline{p})\overline{\eta}\Bigr)
\\
&\quad
+\frac{1}{2N\delta}
\Bigl(
\mathbf{n}(p)^2\partial_{z}^2f(p,\overline{p})\zeta^2
+
\overline{\mathbf{n}(p)}^2\partial_{\overline{w}}^2f(p,\overline{p})\overline{\eta}^2
\Bigr)
+\frac{|\mathbf{n}(p)|^2}{N\delta}\partial_{z}\partial_{\overline{w}}f(p,\overline{p})\zeta\overline{\eta}
+O\Bigl(\frac{1}{N^{3/2}}\Bigr).
\end{split}
\end{align}
By Lemma~\ref{Lem_OPOV}, for a given $p\in\partial S$, we have     
\[
f(p,\overline{p})=2\re\Bigl[\frac{\sqrt{p\overline{p}}-\tau p}{1-\tau^2}- g_{\tau}(p)\Bigr]=\Omega(p)=0. 
\]
Note that by \eqref{Schwarz1} and \eqref{Schwarz2}, we have $g_{\tau}'(z)=C(z)$. 
Furthermore, since $\overline{p}=\mathcal{S}(p)$ for $p\in\partial S$, we have 
\[
0
=
A\overline{p}-\sqrt{p\overline{p}}(2g_{\tau}'(p)+B)
=
2\Bigl(\frac{\overline{p}-\tau|p|}{1-\tau^2}-|p| g_{\tau}'(p)\Bigr).
\]
Therefore, we obtain
\begin{equation}
    \label{Derivative f z}
\partial_z f(z,\overline{w})|_{z=p,w=p}
=
\frac{1}{1-\tau^2}\frac{\overline{p}-\tau|p|}{|p|}-g_{\tau}'(p)
=0.
\end{equation}
Next, we compute $\partial_z^2f(z,\overline{w})|_{z=p,w=p}$.
Note that 
the outward unit normal vector can be written in terms of $\psi$ as
\begin{equation} \label{normal vector psi}
\mathbf{n}(p)=\frac{\psi(p)}{\psi'(p)}|\psi'(p)|,
\qquad p\in\partial S,
\end{equation}
see e.g. \cite[Appendix C]{LR16}. 
Notice also that 
\begin{equation}\label{Relationship normal Schwartz}
\mathcal{S}'(p)=-\frac{\psi'(p)}{\overline{\psi'(p)}\psi(p)^2}=-\frac{1}{\mathbf{n}(p)^2},
\qquad p\in\partial S,
\end{equation}
see e.g. \cite[Eq.(30)]{LR16}.
By taking the square root of \eqref{Schwarz2} and differentiating it with respect to $z$, 
we have 
\begin{equation}
\label{Derivative gtau 2}
g_{\tau}''(z)
=
\frac{A}{4} \frac{1}{\sqrt{z}} \frac{\mathcal{S}'(z)}{\sqrt{\mathcal{S}(z)}}
-
\frac{1}{4z} \bigl(2g_{\tau}'(z)+B \bigr).
\end{equation}
Therefore, by \eqref{Derivative f z}, \eqref{Relationship normal Schwartz}, and \eqref{Derivative gtau 2},
we have 
\begin{equation}
\label{Derivative f z2}
\partial_z^2f(z,\overline{w})|_{z=p,w=p}
=
\frac{1}{2(1-\tau^2)}\frac{|p|}{p^2}-g_{\tau}''(p)
=
-\delta\, \mathcal{S}'(p)=\frac{\delta}{\mathbf{n}(p)^2}.
\end{equation}
Finally, note that $|\mathbf{n}(p)|=1$ and 
\begin{equation}
    \label{Derivative f zw}
\partial_z\partial_{\overline{w}}f(z,\overline{w})|_{z=p,w=p}   
=
\frac{1}{2|p|(1-\tau^2)}=\delta.
\end{equation}
By combining \eqref{Taylor fzw} with \eqref{Derivative f z}, \eqref{Derivative f z2}, and \eqref{Derivative f zw}, we obtain \eqref{EFNA}.
This completes the proof. 
\end{proof}

We now complete the proof of \eqref{GinUE edge kernel} in Theorem~\ref{Thm_StrongNonHermiticity}.

\begin{proof}[Proof of \eqref{GinUE edge kernel} in Theorem~\ref{Thm_StrongNonHermiticity}]

Using  Proposition~\ref{Prop_Det Pfa struc} (i), \eqref{RNk rescaled complex} and \eqref{RescaleKernelS}, it suffices to derive the scaling limit of the kernel
\begin{equation}
\label{HatKc}
    K_N^{(\mathrm{c})}(\zeta,\eta)
    := \sqrt{\frac{2AN}{\pi}} 
\frac{\sqrt{\omega_N^{{\rm c}}(z)\omega_N^{{\rm c}}(\overline{w})}}{(z\overline{w})^{\frac{\nu}{2}-\frac{1}{4}}
\exp\bigl(-N\frac{2\sqrt{z\overline{w}}-\tau(z+\overline{w})}{1-\tau^2}\bigr)}
\widehat{S}_N^{(\nu)}(\zeta,\overline{\eta}) . 
\end{equation}
It suffices to compute the asymptotic behaviour of the terms in  \eqref{WhatIN} involving Laguerre polynomials. 
The remaining part follows from straightforward computations using the Taylor expansion and Stirling's formula. 

Recall that $\phi$ is given by \eqref{ConformalMap1}. 
We use the parametrization of the boundary $\gamma_\tau:[0,2\pi)\to \C$: 
\[
 \gamma_\tau(\theta)=\phi(e^{i\theta})=e^{i\theta}+\tau^2e^{-i\theta}+2\tau.
\]
Then it follows from 
$$
\phi'(e^{i\theta})=1-\tau^2e^{-2i\theta}, \qquad \psi(\phi(u))=u, \qquad \psi'(\phi(e^{i\theta}))=1/\phi'(e^{i\theta})=1/(1-\tau^2e^{-2i\theta})
$$
that 
\[
\mathbf{n}(\gamma_\tau(\theta))
=\frac{e^{i\theta}}{\sqrt{(1-\tau^2e^{-2i\theta})(1-\tau^2e^{2i\theta})}}(1-\tau^2e^{-2i\theta})
=\frac{e^{i\theta}-\tau^2e^{-i\theta}}{\sqrt{1-2\tau^2\cos 2\theta+\tau^4}}. 
\]
Therefore, by \eqref{normal vector psi}, we obtain 
\begin{equation}
\label{NPsi_01}
\mathbf{n}(p)\Bigl( \frac{|\psi'(p)|}{\psi(p)}-\tau^2\frac{|\psi'(p)|}{\overline{\psi(p)}} \Bigr)
=
\frac{(e^{i\theta}-\tau^2e^{-i\theta})(e^{-i\theta} -\tau^2e^{i\theta})}{1-2\tau^2\cos 2\theta+\tau^4}
=1. 
\end{equation} 

By using Lemma~\ref{Lem_Strong1}, we obtain
\begin{align*}
&\quad L_{N-1}^{(\nu)}\Bigl(\frac{N}{\tau}z \Bigr)L_{N}^{(\nu)}\Bigl(\frac{N}{\tau}\overline{w} \Bigr)
-
\tau^2 L_{N}^{(\nu)}\Bigl(\frac{N}{\tau}z \Bigr)L_{N-1}^{(\nu)}\Bigl(\frac{N}{\tau}\overline{w} \Bigr)
\\
&=\frac{1}{2\pi N} \frac{(-1)^{2N-1}}{\tau^{2N-1}} 
\frac{e^{N(g_{\tau}(z)+g_{\tau}(\overline{w}))}}{|p|^{\nu}}
\Bigl(\frac{|\psi'(p)|}{\psi(p)}-\tau^2\frac{|\psi'(p)|}{\overline{\psi(p)}}\Bigr)
\Bigl( 1+O\Bigl(\frac{1}{\sqrt{N}}\Bigr)\Bigr)
\\
&=
\frac{1}{\mathbf{n}(p)}
\frac{1}{2\pi N} \frac{(-1)^{2N-1}}{\tau^{2N-1}} 
\frac{e^{N(g_{\tau}(z)+g_{\tau}(\overline{w}))}}{|p|^{\nu}}
\Bigl( 1+O\Bigl(\frac{1}{\sqrt{N}}\Bigr)\Bigr),
\end{align*}
where we used $|\psi(p)|=1$ for $p\in\partial S$ and \eqref{NPsi_01}. 
Then by \eqref{EFNA} and \eqref{RescaledODES}, we obtain 
\[
\partial_{\zeta}\widehat{S}_{N}^{(\nu)}(\zeta,\overline{\eta})
=
-\frac{1}{2}\sqrt{\frac{2}{\pi}}e^{-\frac{1}{2}(\zeta+\overline{\eta})^2}
\Bigl( 1+O\Bigl(\frac{1}{\sqrt{N}}\Bigr)\Bigr). 
\]
By integrating the above with respect to $\zeta$ and using the boundary condition $\widehat{S}_{N}^{(\nu)}(\zeta,\overline{\eta})\to0$ as $\zeta\to\infty$, we have
\begin{equation}
\label{UEcase}
\widehat{S}_{N}^{(\nu)}(\zeta,\overline{\eta})=\frac{1}{2}\erfc\Bigl(\frac{\zeta+\overline{\eta}}{\sqrt{2}}\Bigr)
\Bigl( 1+O\Bigl(\frac{1}{\sqrt{N}}\Bigr)\Bigr). 
\end{equation}
Finally, note that by \eqref{omega N asymp}, we have
\begin{equation}
\label{Asymp_Weight}
\sqrt{\frac{2AN}{\pi}} 
\frac{\sqrt{\omega_N^{{\rm c}}(z)\omega_N^{{\rm c}}(\overline{w})}}{(z\overline{w})^{\frac{\nu}{2}-\frac{1}{4}}
\exp\bigl(-N\frac{2\sqrt{z\overline{w}}-\tau(z+\overline{w})}{1-\tau^2}\bigr)}
=
c_N(\zeta,\eta)
e^{-\frac{1}{2}(|\zeta|^2+|\eta|^2)+\zeta\overline{\eta}}
\Bigl( 1+O\Bigl(\frac{1}{N}\Bigr)\Bigr), 
\end{equation}
where $c_N(\zeta,\eta)$ is a cocycle factor. 
Therefore, one can see that $K_N^{\rm (c)}$ in \eqref{HatKc} converges to $K_{\rm e}$ in \eqref{GinUE edge kernel} as $N\to\infty$, uniformly for $\zeta,\eta$ in compact subsets of $\C$. 
This completes the proof.
\end{proof}


\subsection{Proof of Theorem~\ref{Thm_StrongNonHermiticity} (ii), the edge case} \label{Subsec_Thm strong (ii) edge}


Let us denote 
\begin{align}
\begin{split} 
\label{RescaleKernelK}
\widehat{\kappa}_N^{(\nu)}(\zeta,\eta)
&:= \sqrt{\frac{\pi}{4AN}}
\frac{(zw)^{\nu-\frac{1}{4}}}{(N\delta)^{3/2}}
\exp\Big(-\frac{4N}{1-\tau^2}\sqrt{zw}+\frac{2N\tau}{1-\tau^2}(z+w) \Big) \bfkappa_N(z,w),
\end{split}
\end{align}
where $\bfkappa_N$ is given by \eqref{def of KN s} and $\delta$ is given by \eqref{def of delta Lap Q}. 
By Theorem~\ref{Thm_ODE} (ii), it follows that 
\begin{equation} 
\label{RescaledODEK}
\mathfrak{D}_{N,m}[\zeta] \, \widehat{\kappa}_N^{(\nu)}(\zeta,\eta)= \frac{|p|\mathbf{n}(p)}{\sqrt{zw}}
\Bigl(
\widehat{S}_{2N}^{(2\nu)}(\zeta,\eta)
-
\widehat{T}_{N}(\zeta,\eta)
\Bigr),
\end{equation}
where
\begin{align}
\begin{split} 
\mathfrak{D}_{N,m}[\zeta] & = \frac{1-\tau^2}{4N}\sqrt{\frac{z}{w}}\frac{\sqrt{N\delta}}{\mathbf{n}(p)}\partial_\zeta^2 +
\Bigl(1+\frac{3(1-\tau^2)}{8N\sqrt{zw}} \Bigr)\partial_{\zeta} 
\\
&\quad - \Bigl( 
\frac{2\mathbf{n}(p)^2|p|(\zeta-\eta)}{\sqrt{zw}}-\frac{\mathbf{n}(p)}{2}\sqrt{\frac{1}{N\delta z^2}}+\frac{\mathbf{n}(p)(16\nu^2-1)(1-\tau^2)}{64N\sqrt{N\delta zw}z}
\Bigr).
\end{split}
\end{align}
Here, 
\begin{equation}
\label{RescaleKernelT}
\widehat{T}_N(\zeta,\eta)
:=
\exp
\Bigl(-\frac{2N}{1-\tau^2}(2\sqrt{zw}-\sqrt{z^2}-\sqrt{w^2}) \Bigr)
\widehat{E}_N(\zeta,\eta),
\end{equation}
where 
\begin{align}
\begin{split}
\label{RescaleKernelE}
\widehat{E}_N(\zeta,\eta)
 & :=
 \frac{1}{N\delta} \sqrt{\frac{\pi}{4AN}}
 (zw)^{\nu-\frac{1}{4}}\exp\Bigl(-\frac{2N}{1-\tau^2}(\sqrt{z^2}+\sqrt{w^2}-\tau (z+ w)) \Bigr)
\\
&\quad \times \frac{\sqrt{\pi}(2N)^{2\nu+2}(2N)!! \tau^{2N}}{2^{N+2\nu-1}(1-\tau^2)\Gamma(N+\nu+1/2)}
L_{2N}^{(2\nu)}\Bigl(\frac{2N}{\tau}z\Bigr)
g_N(w). 
\end{split}
\end{align}
Recall that $g_N$ is given by \eqref{def of gN}.
By Proposition~\ref{Prop_RemainderODE}, we have 
\begin{equation}
\label{RescaledODEEN}
\bigg[ \frac{(1-\tau^2)\sqrt{\delta}}{4\sqrt{N}}\partial_\eta^2
 + \Bigl(1+\frac{3(1-\tau^2)}{8Nw} \Bigr)\partial_{\eta}
+
\frac{1}{\sqrt{N\delta}}
\Bigl(
\frac{3}{4w}-\frac{(16\nu^2-1)(1-\tau^2)}{64Nw^2}
\Bigr)
\bigg]\widehat{E}_N(\zeta,\eta) 
= \widehat{\mathrm{II}}_{N}^{(\nu)}(\zeta,\eta),
\end{equation}
where  
\begin{align}
    \begin{split}
    \label{WhatIIN}
        \widehat{\mathrm{II}}_{N}^{(\nu)}(\zeta,\eta)
        & :=
        \sqrt{\frac{\pi}{4AN}}
\frac{(2N)^{2\nu+2}\tau^{4N-1}}{w(1-\tau^2)(N\delta)^{3/2}}
\frac{(2N)!}{\Gamma(2N+2\nu)}
(zw)^{\nu-\frac{1}{4}}
e^{-\frac{2N}{1+\tau}(z+w)}
L_{2N}^{(2\nu)}\Bigl(\frac{2N}{\tau}z \Bigr)L_{2N-1}^{(2\nu)}\Bigl(\frac{2N}{\tau}w \Bigr).
    \end{split}
\end{align}

\begin{lem}\label{Lem_ES2}
As $N\to\infty$, we have
\begin{equation}
\widehat{T}_N(\zeta,\eta)=\frac{e^{(\zeta-\eta)^2}}{\sqrt{2}}e^{-2\zeta^2}\erfc(\sqrt{2}\eta)
\Bigl( 1+O\Bigl(\frac{1}{\sqrt{N}}\Bigr)\Bigr) 
\end{equation}
uniformly for $\zeta,\eta$ in compact subsets of $\C$. 
\end{lem}

\begin{proof}
By Stirling formula, we have 
\[
\sqrt{\frac{\pi}{4AN}}
\frac{(2N)^{2\nu+2}\tau^{4N-1}}{(1-\tau^2)(N\delta)^{3/2}}
\frac{(2N)!}{\Gamma(2N+2\nu)}
\frac{(zw)^{\nu-\frac{1}{4}}}{w}
=2^3\sqrt{\pi}p^{2\nu}\tau^{4N-1}(1-\tau^2)N\Bigl( 1+O\Bigl(\frac{1}{N}\Bigr)\Bigr).
\]
On the other hand, by using Lemma~\ref{Lem_Strong1}, 
\[
L_{2N}^{(2\nu)}\Bigl(\frac{2N}{\tau}z\Bigr)
L_{2N-1}^{(2\nu)}\Bigl(\frac{2N}{\tau}w \Bigr)
=
-
\frac{1}{N}\frac{1}{2^2\pi} \frac{p^{-2\nu}}{\tau^{4N-1}(1-\tau^2)} 
\exp\Bigl(2N(g_{\tau}(z)+g_{\tau}(w))\Bigr)
\Bigl( 1+O\Bigl(\frac{1}{\sqrt{N}}\Bigr)\Bigr).
\]
Note here that 
\[
\exp\Bigl(-\frac{2N}{1+\tau}(z+w)
+
2N(g_{\tau}(z)+g_{\tau}(w))\Bigr)
=
\exp\bigl(-2(\zeta^2+\eta^2)\bigr)\Bigl( 1+O\Bigl(\frac{1}{N}\Bigr)\Bigr).
\]
Using the above asymptotic expansions, we have 
\[
\partial_{\eta}\widehat{E}_N(\zeta,\eta) 
=-\frac{2}{\sqrt{\pi}}e^{-2(\zeta^2+\eta^2)}
\Bigl( 1+O\Bigl(\frac{1}{\sqrt{N}}\Bigr)\Bigr).
\]
By the boundary condition $\widehat{E}_N(\zeta,\eta)\to0$ as $\eta\to\infty$, 
we have 
\[
\widehat{E}_N(\zeta,\eta) 
=
\frac{1}{\sqrt{2}}e^{-2\zeta^2}\erfc(\sqrt{2}\eta)
\Bigl( 1+O\Bigl(\frac{1}{\sqrt{N}}\Bigr)\Bigr).
\]
Combining all of the above, the lemma follows. 
\end{proof}

\begin{proof}[Proof of \eqref{GinSE edge kernel} in Theorem~\ref{Thm_StrongNonHermiticity}]

By Proposition~\ref{Prop_Det Pfa struc} (i), \eqref{RNk rescaled symplectic} and \eqref{RescaleKernelS}, it is enough to derive the scaling limit of
\begin{equation}
\label{HatKs}
    \kappa_N^{(\mathrm{s})}(\zeta,\eta)
    := \sqrt{\frac{4AN}{\pi}} 
\frac{\sqrt{\omega_N^{{\rm s}}(z)\omega_N^{{\rm s}}(w)}}{(zw)^{\nu-\frac{1}{4}}
\exp\bigl(-2N\frac{2\sqrt{zw}-\tau(z+w)}{1-\tau^2}\bigr)}
\widehat{\kappa}_N^{(\nu)}(\zeta,\eta).
\end{equation}
Note that by \eqref{UEcase}, we have 
\[
\widehat{S}_{2N}^{(2\nu)}(\zeta,\eta)=e^{(\zeta-\eta)^2}\erfc(\zeta+\eta)\Bigl( 1+O\Bigl(\frac{1}{\sqrt{N}}\Bigr)\Bigr),\quad \text{as $N\to\infty$},
\]
uniformly for $\zeta,\eta$ in compact subsets of $\C$. 
Using this together with \eqref{RescaledODES} and Lemma~\ref{Lem_ES2}, we have 
\begin{equation}
\partial_{\zeta}\widehat{\kappa}_N^{(\nu)}(\zeta,\eta) = e^{(\zeta-\eta)^2}\erfc(\zeta+\eta)-\frac{e^{(\zeta-\eta)^2}}{\sqrt{2}}e^{-2\zeta^2}\erfc(\sqrt{2}\eta)
+O\Bigl(\frac{1}{\sqrt{N}}\Bigr). 
\end{equation}
Taking $N \to \infty$, we arrive at the differential equation for $\kappa_{\rm e}$ in \eqref{GinSE edge kernel}, which first appeared in \cite{ABK22} and later used in \cite{BE23}, see also \cite[Propositoin 2.2 (a)]{BF23a}.
Then we can again use the initial condition $\widehat{\kappa}_N^{(\nu)}(\eta,\eta)=0$ to complete the proof. 
\end{proof}

\section{Scaling limits at weak non-Hermiticity}  \label{Section_weak bulk}

In this section, we prove Theorems~\ref{Thm_BulkWeakNonHermiticity} and ~\ref{Thm_EdgeWeakNonHermiticity}. 

\subsection{Proof of Theorem~\ref{Thm_BulkWeakNonHermiticity} (i)}

We shall use the following  Plancherel-Rotach type asymptotic behaviour of the Laguerre polynomial. 

\begin{lem}\label{Lem_LUW}
Let $z = z(\zeta)$ be the local coordinate given by \eqref{LightNotationZ}. 
Then as $N \to \infty$, we have 
\begin{align}
\begin{split}
\label{LaguerreTauBulk1}
L_N^{(\nu)}\Bigl(\frac{N}{\tau}z \Bigr)
&= (-1)^N 
\frac{\exp\bigl(\frac{N}{2}p+c\sqrt{p}\zeta+\frac{c^2p}{2}\bigr)}{p^{\nu/2}\bigl(\frac{\pi}{2}\sqrt{p(4-p)}\bigr)^{1/2}}
\frac{1}{N^{(\nu+1)/2}}\sqrt{\frac{(N+\nu)!}{N!}}
\\
&\quad \times
\cos\Bigl( 
2N\beta(p)-(\nu+1)\arccos\bigl(\frac{\sqrt{p}}{2}\bigr)+\frac{c}{2}\sqrt{4-p}(2\zeta+c\sqrt{p})+\frac{\pi}{4}
\Bigr)
(1+O(N^{-1})),
\end{split}
\\
\begin{split}
    \label{LaguerreTauBulk2}
L_{N-1}^{(\nu)}\Bigl(\frac{N}{\tau}z \Bigr)
&= (-1)^{N-1}
\frac{\exp\bigl(\frac{p}{2}N+c\sqrt{p}\zeta+\frac{c^2p}{2}\bigr)}{p^{\nu/2}(\frac{\pi}{2}\sqrt{(4-p)p})^{1/2}}
\frac{1}{N^{(\nu+1)/2}}\sqrt{\frac{(N+\nu)!}{N!}}
\\
&\quad \times
\cos\Bigl( 
2N\beta(p)-(\nu-1)\arccos\bigl(\frac{\sqrt{p}}{2}\bigr)+\frac{c}{2}\sqrt{4-p}(2\zeta+c\sqrt{p})+\frac{\pi}{4}
\Bigr)
(1+O(N^{-1})),
\end{split}
\end{align}
uniformly for $\zeta$ in a compact subset of $\C$, where  
\begin{equation}
\beta(p)=\frac{\sqrt{p(4-p)}}{4}-\arccos\bigl(\frac{\sqrt{p}}{2}\bigr), \qquad p \in [0,4].  
\end{equation}
\end{lem}
\begin{proof}
By \eqref{LightNotationZ} and Taylor expansion, 
we have 
\[
z=p+\frac{2c\sqrt{p}}{N}\zeta-\frac{c^3\sqrt{p}}{2N^2}\zeta-\frac{c^5\sqrt{p}}{16N^3}\zeta+O(N^{-4}).
\]
Therefore it follows that 
\begin{align}
\label{ZE1}
\frac{N}{\tau}z&=4N\cdot \frac{1}{4}\Bigl(p+\frac{2c\sqrt{p}\zeta+c^2p}{N}+O(N^{-2})\Bigr),\\
\label{ZE2}
\frac{N}{\tau}z&=4(N-1)\cdot \frac{1}{4}\Bigl(p+\frac{2c\sqrt{p}\zeta+(c^2+1)p}{N}+O(N^{-2})\Bigr).
\end{align} 
By Lemma~\ref{StrongBulkFixNu} and the behaviours \eqref{ZE1}, and \eqref{ZE2}, we obtain \eqref{LaguerreTauBulk1} and \eqref{LaguerreTauBulk2}.
This completes the proof.
\end{proof}

We now complete the proof of Theorem~\ref{Thm_BulkWeakNonHermiticity} (i).

\begin{proof}[Proof of Theorem~\ref{Thm_BulkWeakNonHermiticity} (i)]

Under the translation invariance along the horizontal direction, the weakly non-Hermitian bulk scaling limit was characterised in \cite[Theorem 3.8]{AB23}, see also \cite{AKMW20}. 
In our present case, the translation invariant limiting kernel is of the form \eqref{GinUE wH bulk kernel}. 
Such a characterisation is particularly helpful, but in general, there is no general theory on verifying the translation invariance.

We shall apply the differential equation in Theorem~\ref{Thm_ODE} to prove the translation invariance. 
First, note that by direct computations, we have 
\[
\frac{\sqrt{\omega_N^{(c)}(z)\omega_N^{(c)}(w)}}{\sqrt{\frac{\pi}{2AN}}(zw)^{\frac{\nu}{2}-\frac{1}{4}}  \exp\bigl(-N\frac{2\sqrt{zw}-\tau(z+w)}{1-\tau^2}\bigr) }
 = c_N(\zeta,\eta) e^{ \zeta \eta- (|\zeta^2|+|\eta|^2)/2   }
 \bigl(1+O(N^{-1})\bigr), 
\]
where $c_N$ is a cocycle. 
By Lemma~\ref{LaguerreTauBulk1} and the elementary trigonometric identity, we have 
\begin{align*}
&\quad 
L_{N-1}^{(\nu)}\Bigl(\frac{N}{\tau}z\Bigr)L_{N}^{(\nu)}\Bigl(\frac{N}{\tau}w \Bigr)
- \tau^2L_{N}^{(\nu)}\Bigl(\frac{N}{\tau}z \Bigr)L_{N-1}^{(\nu)}\Bigl(\frac{N}{\tau}w \Bigr)
\\
&=
\frac{1}{\pi} \frac{1}{p^{\nu}}
\frac{\exp\bigl(Np+c\sqrt{p}(\zeta+\eta)+c^2p\bigr)}{N^{\nu}} 
\frac{\Gamma(N+\nu)}{N!} 
\Bigl( \sin\bigl(c\sqrt{4-p}(\zeta-\eta) \bigr) + O(N^{-1}) \Bigr),
\end{align*}
as $N\to \infty.$
Note that by Taylor expansion, we have 
\[
\exp\Bigl(-\frac{N}{1-\tau^2}(2\sqrt{zw}-\tau(z+w))\Bigr)
=
\exp\Bigl(-Np-\frac{cp^2}{2}+\frac{1}{2}(\zeta-\eta)^2-c\sqrt{p}(\zeta+\eta)
\Bigr)(1+O(N^{-1}))
\]
and 
\begin{align*}
\sqrt{\frac{\pi}{2AN}}
\frac{N^{\nu+2}\tau^{2N-1} }{(1-\tau^2)^2(N\delta)^{3/2}} \frac{N!}{\Gamma(N+\nu)}
\frac{(zw)^{\frac{\nu}{2}+\frac{1}{4}}}{zw}
=
\frac{\sqrt{2\pi} N^{\nu} \cdot N!}{\Gamma(N+\nu)}
  p^{\nu}e^{-2c^2}
\bigl(1+O(N^{-1})\bigr). 
\end{align*}
Therefore, by combining the above asymptotic expansions with \eqref{RescaledODES} and \eqref{WhatIN}, 
we obtain 
\[
\partial_{\zeta}\widehat{S}_N^{(\nu)}(\zeta,\eta)
=
 \sqrt{\frac{2}{\pi}}\,e^{-\frac{c^2}{2}(4-p)+\frac{1}{2}(\zeta-\eta)^2}\sin\bigl(c\sqrt{4-p}(\zeta-\eta) \bigr)+O(N^{-1}),\quad\text{as $N\to\infty$},
\]
uniformly for $\zeta,\eta$ in compact subsets of $\C$. Note that the right-hand side of this equation depends on $\zeta-\eta$, which provides the translation invariance along the horizontal direction. This completes the proof. 
\end{proof}

\subsection{Proof of Theorem~\ref{Thm_BulkWeakNonHermiticity} (ii)}

In this subsection, we prove Theorem~\ref{Thm_BulkWeakNonHermiticity} (ii).
As in the previous subsection, we need the following asymptotic behaviours, which again follow from \eqref{ZE1}, \eqref{ZE2}, and Lemma~\ref{StrongBulkFixNu}. 

\begin{lem}\label{Lem_LSW1}
Let $z = z(\zeta)$ be the local coordinate given by \eqref{LightNotationZ}.
Then as $N\to\infty$, we have 
\begin{align}
\begin{split}
\label{SBulkcase1}
L_{2N}^{(2\nu)}\Bigl(\frac{2N}{\tau}z \Bigr)
&= \frac{e^{Np+(2c\sqrt{p}\zeta+c^2p)}}{(2p)^{\nu}(\pi\sqrt{p(4-p)})^{1/2}}\frac{1}{N^{(2\nu+1)/2}}
\sqrt{\frac{(2N+2\nu)!}{(2N)!}}
\\
&\quad
\times 
\cos\Bigl(
4N\beta(p)-(2\nu+1)\arccos\frac{\sqrt{p}}{2}+c\sqrt{4-p}(2\zeta+c\sqrt{p})+\frac{\pi}{4}
\Bigr)(1+O(N^{-1})),
\end{split}
\\
\begin{split}
\label{SBulkcase3}
L_{2N-1}^{(2\nu)}\Bigl(\frac{2N}{\tau}z\Bigr) 
&=
\frac{e^{Np+(2c\sqrt{p}\zeta+c^2p)}}{(\pi\sqrt{p(4-p)})^{1/2}} (2Np)^{-\nu}N^{-1/2}
\sqrt{\frac{(2N+2\nu-1)!}{(2N-1)!}}
\\
&\quad
\times 
\cos\Bigl(
4N\beta(p)-(2\nu-1)\arccos\frac{\sqrt{p}}{2}+c\sqrt{4-p}(2\zeta+c\sqrt{p})+\frac{\pi}{4}
\Bigr)(1+O(N^{-1})),
\end{split}
\end{align}
uniformly in $\zeta$ in a compact subset of $\C$. 
\end{lem}

For the proof of Theorem~\ref{Thm_BulkWeakNonHermiticity} (ii), we shall again make use of the differential equation \eqref{RescaleKernelT}.
For this purpose, one needs to derive the asymptotic behaviours of \eqref{RescaleKernelS} and \eqref{RescaleKernelT}. 

\begin{lem}\label{Lem_BE2w}
As $N\to\infty$, we have 
\begin{equation}
\label{hatENODE3}
    \widehat{E}_N(\zeta,\eta)\to 0,
\end{equation}
uniformly for $\zeta,\eta$ in compact subsets of $\C$.
\end{lem}


\begin{proof}[Proof of Lemma~\ref{Lem_BE2w}]
By combining Lemma~\ref{Lem_LSW1} and Stirling's formula, after some computations, we have
\begin{align}
\begin{split}
\label{hatENODE1}
\partial_{\eta}\widehat{E}_N(\zeta,\eta)&=
\frac{1}{N}
\frac{8c^2}{\sqrt{\pi}} \frac{e^{-c^2(4-p)}}{\sqrt{p(4-p)}}
\bigg( 
\cos\bigl(2c\sqrt{4-p}(\zeta-\eta)-2\arccos\frac{\sqrt{p}}{2} \bigr)
\\
&\quad
-\sin\bigl( 
2c\sqrt{4-p}(\zeta+\eta)+2(c^2+N)\sqrt{p(4-p)}-4(2N+\nu)\arccos\frac{\sqrt{p}}{2}
\bigr)
\bigg) 
(1+O(N^{-1})),
\end{split}
\end{align}
uniformly for $\zeta,\eta$ in compact subsets of $\C$.
Next, we show that as $N\to\infty$, 
\begin{align}
\label{hatENODE2}
\widehat{E}_N(\zeta,0)&\to 0,
\end{align}
uniformly for $\zeta,\eta$ in compact subsets of $\C$. 
We write 
\begin{equation}
\label{ENODE1}
\mathcal{E}_N(p)
=
N^{\nu}p^{\nu+1/2}e^{-\frac{2N}{1+\tau}p}g_N(p),
\end{equation}
where $g_N(p)$ is defined by \eqref{def of gN}.
Then one can express \eqref{RescaleKernelE} as 
\begin{equation}
\label{What EN 0}
\widehat{E}_N(\zeta,0)
=
\sqrt{\frac{\pi}{4AN}} \frac{z^{\nu-\frac{1}{4}}e^{-\frac{2N}{1+\tau}z}}{p^{\frac{3}{4}}N\delta}
\frac{\sqrt{\pi}(2N)^{\nu+2}(2N)!! \tau^{2N}}{2^{N+\nu-1}(1-\tau^2)\Gamma(N+\nu+1/2)}
L_{2N}^{(2\nu)}\Bigl(\frac{2N}{\tau}z\Bigr)
\mathcal{E}_N(p).
\end{equation}
Note that by using \eqref{SBulkcase1}, after straightforward computations, one can show that the pre-factor of  $\mathcal{E}_N(p)$ in the right-hand side of this identity is of order $O(1)$. 
Thus it suffices to show that $\mathcal{E}_N(p)=o(1)$. 
By Proposition~\ref{Prop_RemainderODE}, we have
\begin{equation}
\label{ENODE2}
\Bigl[\frac{1-\tau^2}{4N}\partial_s^2
+
\partial_s
-
\frac{(1-\tau^2)(4\nu^2-1)}{16Ns^2}
\Bigr]\mathcal{E}_N(s)
=
s^{\nu-\frac{1}{2}}e^{-\frac{2N}{1+\tau}s}
N^{\nu}\frac{(2N-1)!!\tau^{2N-1}}{2^{N}\Gamma(N+\nu)}
L_{2N-1}^{(2\nu)}\Bigl(\frac{2N}{\tau}s\Bigr).
\end{equation}
Here, it follows from  Lemma~\ref{StrongBulkFixNu} that for $s\in(0,p)$ with $p\in(0,4)$, 
\begin{equation}
\label{cEN asym Bulk weak}
\partial_s \mathcal{E}_N(s)
= \frac{e^{-2c^2+\frac{c^2}{2}s}}{\pi s^{3/4}(4-s)^{1/4}}
\cos\Bigl( 
4N\beta(s)+c^2\sqrt{s(4-s)}-(2\nu-1)\arccos\frac{\sqrt{s}}{2}+\frac{\pi}{4}
\Bigr)
\bigl( 1+O(N^{-1}) \bigr).
\end{equation}
Since $\mathcal{E}_N(0)=0$ for $\nu>-1/2$ by \eqref{ENODE1}, we can write
\begin{equation}
\label{ENODE3}
\mathcal{E}_N(p) 
=\int_{0}^p\partial_s \mathcal{E}_N(s)ds. 
\end{equation}
We now estimate the oscillatory integral \eqref{ENODE3} following \cite[Lemma 4.4]{BKLL23}.
Let 
\[
\mathfrak{a}(t)
:=
\frac{1}{\pi}\frac{1}{t^{3/4}(4-t)^{1/4}}
\exp\Bigl( 
-2c^2+\frac{c^2}{2}t-i\bigl(c^2\sqrt{t(4-t)}-(2\nu-1)\arccos\frac{\sqrt{t}}{2}+\frac{\pi}{4}\bigr)
\Bigr)
\]
and 
\[
\widetilde{\phi}(t):=
\sqrt{t(4-t)}-4\arccos\frac{\sqrt{t}}{2}.
\]
Then, by combining \eqref{ENODE3} with \eqref{cEN asym Bulk weak}, we have 
\[
\mathcal{E}_N(p)\sim \re \int_0^p \mathfrak{a}(t) e^{iN\widetilde{\phi}(t)}dt,\quad \text{as $N\to\infty$}. 
\]
Note $\widetilde{\phi}'(t)\neq 0$ for $t\in(0,4)$ with $\widetilde{\phi}'(4)=0$ and $0<t<p<4$. 
Notice also that $1/(t^{3/4}(4-t)^{1/4})$ is $L^1$-integrable over $[0,4]$. 
Therefore, by Riemann-Lebesgue lemma, we have 
\[
\mathcal{E}_N(p)\sim \re \int_0^p \mathfrak{a}(t) e^{iN\widetilde{\phi}(t)} \,dt=O(N^{-1}),\quad \text{as $N\to\infty$}. 
\]
This shows \eqref{hatENODE2}.
By combining \eqref{hatENODE1} with \eqref{hatENODE2}, we obtain \eqref{hatENODE3}, which completes the proof. 
\end{proof}

We are ready to complete the proof of Theorem~\ref{Thm_BulkWeakNonHermiticity} (ii). 

\begin{proof}[Proof of Theorem~\ref{Thm_BulkWeakNonHermiticity} (ii)]
By \eqref{RescaleKernelT} and Lemma~\ref{Lem_BE2w}, we have 
\[
\widehat{T}_N(\zeta,\eta)
= e^{-(\zeta-\eta)^2}  O\Bigl(\frac{1}{N}\Bigr),\quad \text{as $N\to\infty$},
\]
uniformly for $\zeta,\eta$ in compact subsets of $\C$. 
On the other hand, following the computations in Theorem~\ref{Thm_BulkWeakNonHermiticity} (i), we have 
\begin{equation}
\label{LimitBulkODE1}
\partial_{\zeta}\widehat{S}_{2N}^{(2\nu)}(\zeta,\eta)
=
\frac{4}{\sqrt{\pi}}e^{(\zeta-\eta)^2-c^2(4-p)}\sin\bigl(2c\sqrt{4-p}(\zeta-\eta)\bigr)\bigl( 1 + O(N^{-1}) \bigr), \quad
\text{as $N\to\infty$},
\end{equation}
uniformly for $\zeta,\eta$ in compact subsets of $\C$. 
As a consequence, we have  
\begin{equation}
\label{Bulk Weak SHatS}
\widehat{S}_{\mathrm{b}}^{(\rm s)}(\zeta,\eta)
= \lim_{N\to\infty}\widehat{S}_{2N}^{(2\nu)}(\zeta,\eta) = \frac{2}{\sqrt{\pi}}e^{(\zeta-\eta)^2} \int_{E_a}e^{-t^2}\cos\bigl(2t(\zeta-\eta)\bigr)dt,
\end{equation}
where $E_a$ is given by \eqref{Bulk Weak ap}. 
Then by combining all of the above with \eqref{RescaledODEK}, we obtain 
\begin{equation}
\partial_{\zeta}\widehat{\kappa}_N^{(\nu)}(\zeta,\eta)-2(\zeta-\eta)\widehat{\kappa}_N^{(\nu)}(\zeta,\eta)
=
\widehat{S}_{\rm b}^{(\rm s)}(\zeta,\eta)+O\Bigl(\frac{1}{N}\Bigr),
\end{equation}
as $N\to \infty$. 
The resulting differential equation is the one appearing in \cite[Proof of Theorem 2.1]{BES23}. 
Under the initial condition $\widehat{\kappa}_N^{(\nu)}(\zeta,\zeta)=0$, it was shown in \cite{BES23} that its unique solution is given by \eqref{GinSE wH bulk kernel}. 
Then by  \eqref{omega N asymp} (with the subscript $\rm s$), the proof of Theorem~\ref{Thm_BulkWeakNonHermiticity} (ii) is complete.
\end{proof}



\subsection{Proof of Theorem \ref{Thm_EdgeWeakNonHermiticity} (i)}

Throughout this subsection, the parameter $\tau$ is given by \eqref{EdgeWeakTau1}. 
We begin with the critical asymptotic behaviour of the Laguerre polynomial that follows from Lemma~\ref{StrongEdgeFixNu}.

\begin{lem}\label{Lem_LUE1w}
Let $z = z(\zeta)$ be the local coordinate given by \eqref{LightNotationZ}. 
Then as $N \to \infty$, we have 
\begin{align}
\begin{split}
\label{UWEdgecase1}
L_{N}^{(\nu)}\Bigl(N\frac{z}{\tau} \Bigr)
&=
(-1)^N
\exp\Bigl(2N+(2N)^{1/3}\bigl(\sqrt{2}c\zeta+\frac{c^4}{4} \bigr)+\frac{\sqrt{2}c^3\zeta+c^6}{4}
\Bigr)
\\
&
\quad \times (4N)^{-\nu/2} (2N)^{-1/3}\sqrt{\frac{(N+\nu)!}{N!}}
\Ai\Bigl(\sqrt{2}c\,\zeta+\frac{c^4}{4} \Bigr)
\bigl( 1+ O(N^{-1/3})\bigr),
\end{split}
\\
\begin{split}
\label{UWEdgecase2}
L_{N-1}^{(\nu)}\Bigl(\frac{N}{\tau}z \Bigr) 
&=
(-1)^{N-1}
\exp\Bigl(2N+(2N)^{1/3}\bigl(\sqrt{2}c\zeta+\frac{c^4}{4} \bigr)+\frac{\sqrt{2}c^3\zeta+c^6}{2}\Bigr)
\\
&\quad\times
(4 N)^{-\nu/2}
(2N)^{-1/3}
\sqrt{\frac{(N+\nu)!}{N!}}
\Ai\Bigl(\sqrt{2}c\zeta+\frac{c^4}{4}\Bigr)
\bigl( 1+ O(N^{-1/3})\bigr),
\end{split}
\end{align}
uniformly for $\zeta$ in compact subsets of $\C$.
\end{lem}

We mention that while the edge scaling limit for the elliptic GinUE was obtained in \cite{Be10}, a much simpler proof is given in \cite[Section 6.1]{AB23}. 
In \cite{AB23}, the generalised Christoffel-Darboux formula from \cite{LR16} was used.
Similarly, with the help of Theorem~\ref{Thm_ODE}, we can also derive the weakly non-Hermitian edge scaling limit.

\begin{proof}[Proof of Theorem \ref{Thm_EdgeWeakNonHermiticity} (i)]
We first note that as $N\to\infty$, 
\begin{equation*}
\frac{\sqrt{\omega_N^{(c)}(z)\omega_N^{(c)}(w)}}{\sqrt{\frac{\pi}{2AN}}(zw)^{\frac{\nu}{2}-\frac{1}{4}}
\exp\bigl(-N\frac{2\sqrt{zw}-\tau(z+w)}{1-\tau^2}\bigr) }
 =
 c_N(\zeta,\eta)
\exp\Bigl(-\frac{1}{2}(\zeta-\eta)^2-(\im\zeta)^2-(\im\eta)^2\Bigr)
\bigl(1+O(N^{-1/3})\bigr),
\end{equation*}
where $c_N$ is a cocycle. 
By combining Lemma~\ref{Lem_LUE1w} with \eqref{RescaledODES} and  \eqref{WhatIN}, after straightforward computations, we have 
\begin{equation}
\partial_{\zeta}
\widehat{S}_N^{(\nu)}(\zeta,\eta)
=
-c^2\sqrt{2\pi} e^{\frac{1}{2}(\zeta-\eta)^2+\frac{\sqrt{2}}{2}c^3(\zeta+\eta)+\frac{c^6}{6}}
\Ai\Bigl(\sqrt{2}c\zeta+\frac{c^4}{4} \Bigr)\Ai\Bigl(\sqrt{2}c\eta+\frac{c^4}{4} \Bigr) +O(N^{-1/3}).
\end{equation}
Then, by using the initial condition that $\widehat{S}_N^{(\nu)}$ vanishes at infinity, the limiting kernel $\widehat{S}_{\rm e}^{({\rm c})} := \lim_{N\to\infty} \widehat{S}_N^{(\nu)}$ is given by
\begin{equation}
\widehat{S}_{\rm e}^{({\rm c})}(\zeta,\eta)
=
2c^2\sqrt{2\pi} 
e^{\frac{1}{2}(\zeta-\eta)^2}
\int_{-\infty}^{0}e^{\frac{\sqrt{2}}{2}c^3(\zeta+\eta-2u)+\frac{c^6}{6}}
\Ai\Bigl(\sqrt{2}c(\zeta-u)+\frac{c^4}{4} \Bigr)
\Ai\Bigl(\sqrt{2}c(\eta-u)+\frac{c^4}{4} \Bigr)du.
\end{equation}
Therefore, we conclude \eqref{GinUE wH edge kernel}.
\end{proof}


\subsection{Proof of Theorem \ref{Thm_EdgeWeakNonHermiticity} (ii)}

In this subsection, we prove Theorem \ref{Thm_EdgeWeakNonHermiticity} (ii). 
As before, we begin with the critical asymptotic behaviour. 

\begin{lem}\label{Lem_SEE1w}
Let $z = z(\zeta)$ be the local coordinate given by \eqref{LightNotationZ}. 
Then as $N \to \infty$, we have 
\begin{align}
\begin{split}
\label{SEdgecase1}
L_{2N}^{(2\nu)}\Bigl(\frac{2N}{\tau}z\Bigr)
&=
\exp\Bigl(4N+2(2N)^{1/3}\bigl(\sqrt{2}c\zeta+\frac{c^4}{4} \bigr)+\frac{\sqrt{2}c^3\zeta+c^6}{2}
\Bigr)
\\
&\quad\times
\frac{(4\cdot 2N)^{-\nu}}{2^{2/3}N^{1/3}}
\sqrt{\frac{(2N+2\nu)!}{(2N)!}}
\Ai\Bigl(2^{2/3}\Bigl(\sqrt{2}c\zeta+\frac{c^4}{4}\Bigr) \Bigr)
\bigl( 1+ O(N^{-1/3})\bigr), 
\end{split}
\\
\begin{split}
\label{SEdgecase3}
L_{2N-1}^{(2\nu)}\Bigl(\frac{2N}{\tau}z \Bigr) 
&=
-
\exp\Bigl(4N+2(2N)^{1/3}\bigl(\sqrt{2}c\zeta+\frac{c^4}{4} \bigr)+\frac{\sqrt{2}c^3\zeta+c^6}{2}\Bigr)
\\
&\quad
\times
\frac{(4\cdot 2N)^{-\nu}}{2^{2/3}N^{1/3}}
\sqrt{\frac{(2N+2\nu)!}{(2N)!}}
\Ai\Bigl(2^{2/3}\Bigl(\sqrt{2}c\zeta+\frac{c^4}{4}\Bigr) \Bigr)
\bigl( 1+ O(N^{-1/3})\bigr), 
\end{split}
\end{align}
uniformly for $\zeta$ in a compact subset of $\C$. 
\end{lem}

Recall that $\widehat{E}_N$ is given by \eqref{RescaleKernelE}. 

\begin{lem}\label{Lem_SEE3w}
Let $\widetilde{c}=2^{1/6}c$. Then as $N\to \infty$, we have 
\begin{equation}
    \label{SE Edge Weak 5}
\widehat{E}_N(\zeta,\eta) \to \widehat{E}_{\rm e}^{(\rm s)}(\zeta,\eta)
:= 4\sqrt{\pi}\,
\widetilde{c}^{\, 2}
e^{\widetilde{c}^{\, 2}\zeta}
\int_{-\infty}^{0}e^{\widetilde{c}^{\, 3}(\zeta-s)+\frac{1}{6}\widetilde{c}^{\, 6}}
\Ai\Bigl(2\widetilde{c}(\zeta-s)+\frac{\widetilde{c}^{\,4}}{4}\Bigr)ds, 
\end{equation}
uniformly for $\zeta,\eta$ in compact subsets of $\C$. 
\end{lem}
\begin{proof}[Proof of Lemma~\ref{Lem_SEE3w}]
By using Lemma~\ref{Lem_SEE1w}, after some computations, we have that as $N\to\infty$, 
\begin{align*}
\widehat{\mathrm{II}}_{N}^{(\nu)}(\zeta,\eta)
&=
-\sqrt{\pi}2^{2+1/3}c^2
\exp\Bigl(\sqrt{2}c^3(\zeta+\eta)+\frac{1}{3}c^6\Bigr)
\\
&\quad\times
\Ai\Bigl(2^{2/3}\Bigl(\sqrt{2}c\zeta+\frac{c^4}{4}\Bigr) \Bigr)
\Ai\Bigl(2^{2/3}\Bigl(\sqrt{2}c\eta+\frac{c^4}{4}\Bigr) \Bigr)\bigl( 1+ O(N^{-1/3})\bigr),
\end{align*}
uniformly for $\zeta,\eta$ in compact subsets of $\C$.
By \eqref{RescaledODEEN}, this shows that as $N \to \infty$,
\begin{align}
\begin{split}
\label{SE Edge Weak 3}
\partial_{\eta}\widehat{E}_N(\zeta,\eta) 
&=
-\sqrt{\pi}2^{2+1/3}c^2
\exp\Bigl(\sqrt{2}c^3(\zeta+\eta)+\frac{1}{3}c^6
\Bigr)
\\
&\quad\times
\Ai\Bigl(2^{2/3}\Bigl(\sqrt{2}c\zeta+\frac{c^4}{4}\Bigr) \Bigr)
\Ai\Bigl(2^{2/3}\Bigl(\sqrt{2}c\eta+\frac{c^4}{4}\Bigr) \Bigr)\bigl( 1+ O(N^{-1/3})\bigr),
\end{split}
\end{align}
uniformly for $\zeta,\eta$ in compact subsets of $\C$. 

Next, we show that as $N\to\infty$, 
\begin{align}
\label{SE Edge Weak 4}
\widehat{E}_N(\zeta,0)
&\to
4\cdot 2^{1/3}c^2\sqrt{\pi}
\Ai\Bigl( 2^{2/3}\Bigl(\sqrt{2}c\zeta+\frac{c^4}{4} \Bigr)\Bigr)
\int_{-\infty}^{0}e^{\frac{1}{3}c^6+\sqrt{2}c^3(\zeta-s)}
\Ai\Bigl(2^{2/3}\Bigl(-\sqrt{2}cs+\frac{c^4}{4} \Bigr) \Bigr)ds,
\end{align}
uniformly for $\zeta,\eta$ in compact subsets of $\C$. 
For this purpose, let us recall \eqref{ENODE1}, \eqref{What EN 0}, and \eqref{ENODE2}. 
Note that by Stirling's formula, we have that as $N\to\infty$,
\begin{equation}
\label{WEA3}
\partial_s\mathcal{E}_N(s)
=
\frac{\sqrt{N}}{\sqrt{\pi}}
\tau^{2N-1} 
s^{\nu-\frac{1}{2}}e^{-\frac{2N}{1+\tau}s}
L_{2N-1}^{(2\nu)}\Bigl(\frac{2N}{\tau}s\Bigr)
\Bigl( 1 + O\Bigl(\frac{1}{N}\Bigr) \Bigr).
\end{equation}
By Lemma~\ref{Lem_SEE1w}, we have that as $N\to\infty$, 
\begin{align}
\begin{split}
\label{WEA4}
&\quad 
\sqrt{\frac{\pi}{4AN}} \frac{z^{\nu-\frac{1}{4}}e^{-\frac{2N}{1+\tau}z}}{p^{\frac{3}{4}}N\delta}
\frac{\sqrt{\pi}(2N)^{\nu+2}(2N)!! \tau^{2N}}{2^{N+\nu-1}(1-\tau^2)\Gamma(N+\nu+1/2)}
L_{2N}^{(2\nu)}\Bigl(\frac{2N}{\tau}z\Bigr)
\\
&=
\sqrt{N} 
2^{\frac{13}{6}}\pi
c\exp\Bigl(\sqrt{2}c^3\zeta+\frac{c^6}{6}\Bigr)
\Ai\Bigl( 2^{2/3}\Bigl(\sqrt{2}c\zeta+\frac{c^4}{4} \Bigr)\Bigr)
\Bigl( 1+O\Bigl(\frac{1}{N^{1/3}}\Bigr)\Bigr),
\end{split}
\end{align}
uniformly for $\zeta$ in a compact subset of $\C$. 
We now compute the asymptotic behaviour of $\mathcal{E}_N(p)$ as $N\to\infty$. 
We fix $r_N=p_N+(2N)^{-2/3+\alpha}$ for $0<\alpha<1/3$. 
We take a point at infinity as a boundary condition on $\R$, and write
\[
\mathcal{E}_N(p)
=\mathcal{E}_N(\infty)-\int_{p_N}^{\infty}\mathcal{E}_N'(s)ds
=-\int_{p_N}^{r_N}\mathcal{E}_N'(s)ds-\int_{r_N}^{\infty}\mathcal{E}_N'(s)ds. 
\]
Note that by using Lemma~\ref{Lem_SEE1w}, we have 
\begin{align*}
\int_{p_N}^{r_N}\mathcal{E}_N'(s)ds
&= (2N)^{-2/3}\int_{0}^{(2N)^{\alpha}}\mathcal{E}_N'(p_N+(2N)^{-2/3}s)ds
\\
&\sim - \frac{1}{\sqrt{N}}\frac{1}{\sqrt{\pi}} \frac{\sqrt{2}}{2^{1/3}}c 
\int_{-\infty}^{0}e^{\frac{1}{6}c^6-\sqrt{2}c^3s}
\Ai\Bigl(2^{2/3}\Bigl(-\sqrt{2}cs+\frac{c^4}{4} \Bigr) \Bigr)ds.
\end{align*}
On the other hand, by \eqref{WEA3} and Lemma~\ref{Lem_Strong1}, we have 
\[
\int_{r_N}^{\infty}\mathcal{E}_N'(s)ds
\sim
\int_{r_N}^{\infty}
\frac{(-1)^{2N-1}}{\sqrt{2}\pi} \frac{\psi(s)^{\nu/2-1}\sqrt{\psi'(s)}}{s^{\nu/2}} e^{-N\Omega(s)} \, ds. 
\]
For a given $s\in[r_N,\infty)$ and a sufficiently large $N$, we have
\[
\Bigl|
 \frac{\psi(s)^{\nu/2-1}\sqrt{\psi'(s)}}{s^{\nu/2}}
\Bigr|
\lesssim  \frac{1}{\sqrt{s}} \frac{1}{(s(s-4))^{1/4}}
\lesssim \frac{1}{(s-4)^{1/4}}.
\]
By Lemma~\ref{Lem_OPOV}, for $s\in[r_N,\infty)$, there exists a positive constant $\epsilon>0$ such that 
$\min_{s\in[r_N,\infty]}\Omega(s)>\epsilon>0$. 
Therefore, for a sufficiently large $N$, we have 
\[
\Bigl| \int_{r_N}^{\infty}\mathcal{E}_N'(s)ds \Bigr|
\sim
\Bigl| \int_{r_N}^{\infty}
\frac{(-1)^{2N-1}}{\sqrt{2}\pi} \frac{\psi(s)^{\nu/2-1}\sqrt{\psi'(s)}}{s^{\nu/2}}
e^{-N\Omega(s)}ds\Bigr|
\leq O(e^{-\epsilon N})\cdot \int_{4}^{\infty}\frac{1}{(s-4)^{1/4}}=O(e^{-\epsilon N}). 
\]
Combining all of the above, we obtain \eqref{SE Edge Weak 4}. 

Finally, under $c=\widetilde{c}/2^{1/6}$, the derivative \eqref{SE Edge Weak 3} and the initial condition \eqref{SE Edge Weak 4} are same as \cite[Eqs.(4.8) and (4.9)]{BES23}, which leads to \eqref{SE Edge Weak 5}.
\end{proof}

We are ready to finish the proof of Theorem \ref{Thm_EdgeWeakNonHermiticity} (ii).

\begin{proof}[Proof of Theorem \ref{Thm_EdgeWeakNonHermiticity} (ii)]
Note that by \eqref{LightNotationZ}, we have
$$
\exp\Big(-\frac{2N}{1-\tau^2}(2\sqrt{zw}-(z+w)) \Big)=e^{(\zeta-\eta)^2}(1+O(N^{-1/3})),
$$ 
as $N\to\infty.$
We set $c=\widetilde{c}/2^{1/6}$. Then as in the previous subsection, we have that as $N\to\infty$,
\begin{align*}
\begin{split}
\partial_{\zeta}
\widehat{S}_{2N}^{(2\nu)}(\zeta,\eta)
=
-\sqrt{\pi} 2^{2+1/3}c^2 
e^{(\zeta-\eta)^2+\sqrt{2}c^3+\frac{1}{3}c^6}
\Ai\Bigl(2^{2/3}\Bigl(\sqrt{2}c\zeta+\frac{c^4}{4}\Bigr) \Bigr)
\Ai\Bigl(2^{2/3}\Bigl(\sqrt{2}c\eta+\frac{c^4}{4}\Bigr) \Bigr)\bigl( 1+ O(N^{-1/6})\bigr),
\end{split}
\end{align*}
uniformly for $\zeta,\eta$ in compact subsets of $\C$. 
As a consequence, by combining \eqref{RescaledODEK} and Lemma~\ref{Lem_SEE3w}, we have  
\begin{equation}
\partial_{\zeta}\widehat{\kappa}_N^{(\nu)}(\zeta,\eta)
=
2(\zeta-\eta)\widehat{\kappa}_N^{(\nu)}(\zeta,\eta)
+
\widehat{S}_{\rm e}^{(\rm s)}(\zeta,\eta)
-
e^{(\zeta-\eta)^2}
\widehat{E}_{\rm e}^{(\rm s)}(\zeta,\eta)
+o(1),
\end{equation}
where 
\begin{equation}
    \label{S Edge Weak 2}
\widehat{S}_{\rm e}^{(\rm s)}(\zeta,\eta)
=
8\sqrt{\pi}\widetilde{c}^{\, 2} e^{(\zeta-\eta)^2}
\int_{-\infty}^{0} \Ai\Bigl(2\widetilde{c}(\zeta-u)+\frac{\widetilde{c}^{\,4}}{4} \Bigr)
\Ai\Bigl(2\widetilde{c}(\eta-u)+\frac{\widetilde{c}^{\,4}}{4} \Bigr)du.
\end{equation}
This resulting differential equation for $\widehat{\kappa}_N$ agree with the one in \cite[Theorem 2.2 and Proposition 4.8]{BES23}. 
In particular, we obtain the universal scaling limit \eqref{GinSE wH edge kernel}.
This completes the proof.
\end{proof}

\bibliographystyle{abbrv}

\end{document}